\documentclass[fleqn, 11pt]{amsart}

\usepackage{xcolor}
\RequirePackage{amssymb}
\RequirePackage{comment}
\usepackage{enumerate}
\usepackage{amsmath}
\usepackage{url}

\usepackage{framed,color}\definecolor{shadecolor}{rgb}{0.9,0.8,0.8}

\theoremstyle{plain}
\newtheorem{thm}{Theorem}[section]
\newtheorem*{thm*}{Theorem}
\newtheorem*{mainthm*}{Main Theorem}
\newtheorem{lem}[thm]{Lemma} \newtheorem*{lem*}{Lemma}
 \newtheorem*{claim*}{Claim}
\newtheorem{cor}[thm]{Corollary} \newtheorem*{cor*}{Corollary}
 \newtheorem*{prop*}{Proposition}

\theoremstyle{definition}
\newtheorem{defn}[thm]{Definition} \newtheorem*{defn*}{Definition}
\newtheorem{case}{Case}

\theoremstyle{remark}
\newtheorem{rem}[thm]{Remark} \newtheorem*{rem*}{Remark}
 \newtheorem*{example*}{Example}
 \newtheorem*{conj*}{Conjecture}
 \newtheorem*{question*}{Question}

\newcommand{\Ord}{\mathrm{Ord}}
\newcommand{\forces}{\mathbin{\Vdash}}
\newcommand{\HOD}{\mathrm{HOD}}

\newcommand{\Hom}{\mathrm{Hom}}

\newcommand{\bfPi}{\bf{\Pi}}
\newcommand{\bfSigma}{\bf{\Sigma}}

\newcommand{\ScS}{\mathrm{ScS}}
\newcommand{\uB}{\mathrm{uB}}
\newcommand{\HC}{\mathrm{HC}}

\newcommand{\lb}{\mathrm{lb}}

\newcommand{\posSigmaoneone}{\mathrm{pos}\Sigma^1_1}

\newcommand{\bbR}{\mathbb{R}}
\newcommand{\ZF}{\mathsf{ZF}}
\newcommand{\ZFC}{\mathsf{ZFC}}
\newcommand{\AD}{\mathsf{AD}}
\newcommand{\DC}{\mathsf{DC}}
\newcommand{\ADR}{\mathsf{AD}_{\bbR}}

\newcommand{\Col}{\mathrm{Col}}

\newcommand{\restrict}{\mathord{\restriction}}

\newcommand{\less}{\mathord{<}}

\DeclareMathOperator{\Ult}{Ult}
\DeclareMathOperator{\ms}{ms}
\DeclareMathOperator{\ran}{ran}

\DeclareMathOperator{\p}{p}

\newcommand{\powerset}{\mathcal{P}}

\begin{document}


\title
[A model of $\mathsf{AD}$ in which every set of reals is $\uB$]
{A model of the Axiom of Determinacy in which every set of reals is universally Baire}

\author{Paul B.\ Larson}\thanks{The research of the first author was partially supported by NSF grants  DMS-0801009, DMS-1201494 and DMS-1764320.}
\address{Department of Mathematics\\ Miami University\\ Oxford, Ohio\\ USA}
\email{larsonpb@miamioh.edu}

\author{Grigor Sargsyan}\thanks{The secind  author’s work is funded by the National Science Center, Poland under the Maestro Call, registration number UMO-2023/50/A/ST1/00258.
This work began at The Mathematisches Forschungsinstitut Oberwolfach, where the second author was a Leibnitz Fellow, in the Spring of 2012. Proofs of the main results were completed by the end of 2017.}
\address{Institute of Mathematics\\
Polish Academy of Sciences\\
8 \'{S}niadeckich Street\\ Warsaw, Poland}
\email{gsargsyan@impan.pl}

\author{Trevor M.\ Wilson}\thanks{The third author was supported by NSF grant DMS-1044150.}
\address{Department of Mathematics\\ Miami University\\ Oxford, Ohio\\ USA}
\email{twilson@miamioh.edu}

\begin{abstract}
The consistency of the theory $\mathsf{ZF} + \ADR + {}$``every set of reals is universally Baire'' is proved relative to $\mathsf{ZFC} + {}$``there is a cardinal that is a limit of Woodin cardinals and of strong cardinals.'' The proof is based on the derived model construction, which was used by Woodin to show  that the theory $\mathsf{ZF} + \ADR + {}$``every set of reals is Suslin'' is consistent relative to  $\mathsf{ZFC} + {}$``there is a cardinal $\lambda$ that is a limit of Woodin cardinals and of $\mathord{<}\lambda$-strong cardinals.'' The $\Sigma^2_1$ reflection property of our model is proved using genericity iterations as in Neeman \cite{NeeHandbook} and Steel \cite{SteSTFree}.
\end{abstract}

\maketitle

\section{Introduction}\label{sec-introduction}


Universal Baireness has come to be seen in modern set theory as a sort of master regularity property for sets of real numbers, implying for instance both Lebesgue measurability and the property of Baire. Moreover, assuming the existence of suitable large cardinals, it is a property shared by every set of reals with a sufficiently simple definition. For instance, if there exist proper class many Woodin cardinals, then every set of reals in the inner model $L(\mathbb{R})$ (the smallest transitive model of $\ZF$ containing the reals and the ordinals) is universally Baire in the full universe $V$ (see, for instance, Theorems 3.3.9 and 3.3.13 of \cite{LarStationaryTower}). In this case, $L(\bbR)$ also satisfies the Axiom of Determinacy ($\AD$). 
However, if $\AD$ holds in $L(\bbR)$ then not every set of reals in $L(\mathbb{R})$ satisfies the definition of universal Baireness in $L(\mathbb{R})$, since in this case $L(\mathbb{R})$ does not satisfy the statement that every set of reals is Suslin (as defined in Definition \ref{Suslindef}; see, for instance, Theorems 6.24 and 6.28 of \cite{Larson:Extensions}). In this paper we produce a model of $\ZF + \AD$ in which every set of reals is universally Baire. The large cardinal hypothesis used in our argument has since been shown to be optimal by Sandra M\"{u}ller \cite{Muller:consistency}.

We let $\omega$ denote the set of nonnegative integers, with the discrete topology, and we let $\omega^\omega$ denote the set of $\omega$-sequences from $\omega$, with the product topology. Following set-theoretic convention, we will also denote $\omega^\omega$ by $\mathbb{R}$ and refer to its elements as \emph{reals}, despite the fact that it is homeomorphic to the space of irrational numbers and not to the real line.

The definition of universal Baireness that we will use in this paper is the set-theoretic definition involving trees, which we postpone to Section \ref{sec-Suslin-uB}.
For now, we remark that every universally Baire
set of reals $A \subseteq \omega^\omega$ has the following property:
\begin{itemize}
 \item[$(\ast)$]
 For every topological space $X$ with a regular open base and every continuous function
 $f \colon X \to \omega^\omega$, the preimage $f^{-1}[A]$ has the Baire property in $X$.
\end{itemize}
Note that while every regular topological space has a regular open basis, some Hausdorff spaces do not. The property $(\ast)$ was the original definition of universal Baireness given by Feng, Magidor, and Woodin in  \cite{FenMagWoo}, where it was shown  to be equivalent in $\mathsf{ZFC}$ to the definition we use in this paper \cite[Corollary 2.1(3)]{FenMagWoo}.
We don't know if the equivalence of the two definitions follows from $\ZF$, but Lemma \ref{uBimpliesstar} shows
that the definition we use implies $(\ast)$, so that all sets of reals in the model we will build are universally Baire according to both definitions.

The following is an immediate consequence of our main theorem (for which see Section \ref{amodelsec}). 

\begin{thm}\label{thm-relative-consistency}
 If the theory $\mathsf{ZFC} + {}$``there is a cardinal that is a limit of Woodin cardinals and a limit of strong cardinals'' is consistent, then so is the theory $\mathsf{ZF} + \ADR + {}$``every set of reals is universally Baire.''
\end{thm}

The axiom $\ADR$ asserts the determinacy of all two-player games of length $\omega$ on the real numbers. The relationships between $\AD$, $\ADR$ and $\AD^{+}$ (a technical strengthening of $\AD$ whose definition we give in Section \ref{amodelsec}) are discussed in Section \ref{amodelsec}; see, for instance, \cite{Larson:Extensions} for more background. 


Section \ref{sec-Suslin-uB} of the paper reviews universally Baire sets, and Section \ref{semiscalesec} presents some material on semiscales. The models we consider in this paper are introduced in Section \ref{canmodelsec}. The main theorem of the paper is stated in Section \ref{amodelsec}, which also contains a review of symmetric extensions and homogeneously Suslin sets. The proof of the main theorem is given in Section \ref{amodelsec}, relative to three facts proved in later sections, including the  $\Sigma^{2}_{1}$-reflection property of our model. Sections \ref{sec-R-genericity-iterations} (on genericity iterations) and \ref{sec-F-absoluteness} (on absoluteness) develop results needed for the proof of $\Sigma^{2}_{1}$ reflection, which is given in Section \ref{sec-sigma-2-1-reflection}. Finally, Section \ref{sec-DM-at-limit-of-lt-lambda-strongs} proves a theorem of Woodin (used in the proof of the main theorem) whose proof has not previously appeared in print, and uses the machinery from this proof to fill in the last remaining detail of the proof of the main theorem. 



\begin{rem}\label{rem-woodin-remark}
The existence of a model of $\ZF + \ADR$ in which all sets of reals are universally Baire was independently proved by Hugh Woodin from the stronger large cardinal hypothesis asserting the existence of proper class many Woodin limits of Woodin cardinals. 
From this hypothesis he produced an inner model satisfying $\mathsf{ZF} + \mathsf{DC} + \mathsf{AD}^+ + {}$``$\omega_1$ is supercompact,'' which in turn implies that every set of reals is universally Baire, and therefore that $\ADR$ holds. (In the choiceless context, supercompactness is defined in terms of normal fine measures.)
 Woodin's model is an enlargement of the Chang model $L(\Ord^\omega)$ obtained by adding a predicate for the club filter on $\powerset_{\omega_1}(\lambda^\omega)$ for every ordinal $\lambda$.
 Woodin showed that these predicates restrict to ultrafilters on the model.
 The theory $\mathsf{ZF} + \mathsf{DC} + \mathsf{AD} + {}$``$\omega_1$ is supercompact'' implies that
 the pointclass of Suslin sets of reals is closed under complementation, by a theorem of Martin and Woodin \cite{MarWooWeaklyHom}. Together with $\mathsf{AD}^+$, this implies that every set of reals is Suslin.
 Using the supercompactness of $\omega_{1}$ again to take ultrapowers of trees, one can then show that every set of reals is universally Baire.
\end{rem}

\begin{rem}
	A model of determinacy in which all sets of reals are universally Baire can be shown to exist from the theory $\ZF$ + $\AD^+$ plus the assumption that there is a limit ordinal $\alpha$ such that $\theta_\alpha$ (the $\alpha$-th member of the Solovay sequence; see Section 6.3 of \cite{Larson:Extensions}) is less than $\Theta$.  Assume $V$ is a model of this theory.
	To construct a model of $\AD^+$ in which all sets of reals are universally Baire, let $\Delta$ be the set of reals of Wadge rank less than $\theta_\alpha$. Then $M=_{def}V_{\theta_{\alpha+1}}\cap \HOD_{\Delta}$ is a model of $\ZF+ \ADR$ in which all sets of reals are universally Baire. This follows from the fact that in $V$, every set in $\Delta$ is $\kappa$-homogenously Suslin for every $\kappa<\theta_\alpha$. We do not know if the proof of this fact has appeared in print, but it can be proven by combining the following facts: Theorem 7.5 and Lemma 7.7 of \cite{SteDMT}, and the results of \cite{KKMW}. See also \cite{MarWooWeaklyHom}. We then use Lemma 7.7 of  \cite{SteDMT} to conclude that $M$ is a model of $\ZF+ \ADR$ in which all sets of reals are universally Baire. See also Lemma 7.6 of \cite{SteDMT}.

We also remark that if $\theta_\alpha$ is a regular cardinal of $\HOD$ then $M$ defined above will satisfy the theory $\ADR+``\Theta$ is a regular cardinal" (see \cite{CLSSSZSquares}). By the results of \cite{LSA}, in the minimal model of the Largest Suslin Axiom there is $\alpha$ such that $\theta_\alpha<\Theta$ and $\theta_\alpha$ is regular in $\HOD$. Thus, a model of the theory $\ADR+``\Theta$ is a regular cardinal"+``All sets of reals are universally Baire" can be constructed inside the minimal model of the Largest Suslin Axiom. 

\end{rem}
	 

\section{Trees, Suslin sets, and universally Baire sets}\label{sec-Suslin-uB}

In this section, we work in $\mathsf{ZF}$, without the Axiom of Choice. 
A \emph{tree} on a class $X$ is a set of finite sequences $T \subseteq X^{\mathord{<}\omega}$
that is closed under initial segments.
For such a tree we let $[T]$ denote the set of all branches (infinite chains) of $T$, so that $[T]\subseteq X^\omega$.
Note that a set $A \subseteq X^\omega$ is closed in the $\omega$-fold product of the discrete topology on $X$ if and only if $A = \p[T]$ for some tree $T$ on $X$.
For a tree on a product $X \times Y$, we will identify sequences of pairs with pairs of sequences (always of the same length),
so for such a tree we may write $[T] \subseteq X^\omega \times Y^\omega$.

The trees we consider will usually be trees on the class $\omega \times \Ord$ where $\Ord$ is the class of ordinals.\footnote{A set $T$ is a tree on $\omega \times \Ord$ if and only if it is a tree on $\omega \times \kappa$ for some cardinal $\kappa$.
In this paper, we will usually not be concerned with what $\kappa$ is, so we will speak of trees on $\omega \times \Ord$ for simplicity.}
By the \emph{projection} of such a tree $T$, we mean the projection of $[T]$ onto its first coordinate, which is a set of reals:
\[ \p[T] = \{x \in \omega^\omega : \exists f \in \Ord^\omega\, (x,f) \in [T]\}.\]
An important fact that we will often use without further comment is that membership in the projection of a tree is absolute: if $M$ is a transitive model of $\mathsf{ZF}$ containing a real $x$ and a tree $T$ on $\omega \times \Ord$, then $x \in \p[T]$ if and only if $M \models x \in \p[T]$.
This follows from the absoluteness of wellfoundedness of the tree $T_{x} = \{t \in \Ord^{\mathord{<}\omega} : (x 
\restrict \left|t\right|, t) \in T\}$.
The same fact applies with $V$ and a generic extension of $V$ in place of $M$ and $V$ respectively.

\begin{defn}\label{Suslindef}
 A set of reals $A$ is \emph{Suslin} if $A = \p[T]$ for some $T$ on $\omega \times \Ord$. Given a set $X$, $A$ is $X$-Suslin if it is the projection of a tree on $\omega \times X$.
\end{defn}


Recall that a set of reals $A$ is ${\bf \Sigma}^1_1$ (analytic) if and only if $A = \p[T]$ for some $T$ on $\omega \times \omega$, so the Suslin sets generalize the analytic sets.


For a set $A \subseteq \omega^\omega \times \omega^\omega$, we can define ``$A$ is Suslin'' and ``$A$ is $\kappa$-Suslin'' in terms of trees on $\omega \times \omega \times \Ord$ and $\omega \times \omega \times \kappa$ respectively, and similarly for 
$A \subseteq \omega^\omega \times \omega^\omega \times \omega^\omega$ and so on.
In this paper we will typically state and prove results for sets of reals and leave the straightforward higher-dimensional generalizations to the reader.


%
%

Suslin sets of reals are important objects of study under the Axiom of Determinacy.
Under the Axiom of Choice, however, every set of reals is Suslin (in a useless way).
More generally, every wellordered set of reals $A = \{x_\alpha : \alpha < \kappa\}$ is Suslin as witnessed by the tree \[\{(x_\alpha \restrict n, \bar{\alpha} \restrict n) : \alpha < \kappa \text{ and } n \in \omega\}\] on $\omega \times \kappa$, where $\bar{\alpha}$ denotes the infinite sequence with constant value $\alpha$. 
Universal Baireness is a strengthening of Suslinness that is nontrivial in the presence of the Axiom of Choice.

Instead of the original definition of universal Baireness, which we called $(\ast)$ above,
we will use the following definition. 

\begin{defn}\label{defn-uB}
Let $\mathbb{P}$ be a poset. 
\begin{itemize}
 \item A pair of trees $(T,\tilde{T})$ on $\omega \times \Ord$ is
 \emph{$\mathbb{P}$-absolutely complementing} if $1_{\mathbb{P}} \forces \p[T] = \omega^\omega\setminus \p[\tilde{T}]$.
 \item A set of reals $A$ is \emph{$\mathbb{P}$-Baire} if $A = \p[T]$ for some $\mathbb{P}$-absolutely complementing pair of trees $(T,\tilde{T})$ on $\omega \times \Ord$.
\end{itemize}
A set of reals $A$ is \emph{universally Baire} if it is $\mathbb{P}$-Baire for every poset $\mathbb{P}$. We write $\uB$ for the pointclass $\{A \subseteq \omega^\omega : A \text{ is universally Baire}\}$. 
\end{defn}

\begin{rem}\label{collapserem}
	If a poset $\mathbb{P}$ regularly embeds into a poset $\mathbb{Q}$, then every $\mathbb{Q}$-Baire set of reals is $\mathbb{P}$-Baire. A classical result of McAloon says that every partial order $\mathbb{P}$ regularly embeds into the partial order $\Col(\omega, \mathbb{P})$, which adds a surjection from $\omega$ to $\mathbb{P}$ by finite approximations (see, for instance, the appendix to \cite{LarStationaryTower}). To show that a set of reals is universally Baire, then, it suffices to show that it is $\Col(\omega, Z)$-Baire for every set $Z$. In the context of the Axiom of Choice, the set $Z$ can be taken to be an infinite cardinal. We show in Lemma \ref{lem-ordinal-ac-equivalent} that a similar implication holds in the models we consider in this paper. 
\end{rem}

It is immediate from the definition that every universally Baire set of reals is Suslin, and that the collection of universally Baire sets is closed under complements. 
The proof of Shoenfield's absoluteness theorem shows that all ${\bf \Sigma}^1_1$ (analytic) sets of reals are universally Baire, from which it follows that all ${\bf \Pi}^1_1$ (coanalytic) sets are as well. 
Universal Baireness for more complex sets of reals is tied to large cardinals.
For example, every $\bfSigma^1_2$ set of reals is universally Baire if and only if every set has a sharp, as shown by by Feng, Magidor, and Woodin  \cite[Theorem~3.4]{FenMagWoo}.
As noted above, if there is a proper class of Woodin cardinals, then every set of reals in $L(\mathbb{R})$ is universally Baire in $V$.
%

Lemma \ref{uBimpliesstar} below shows that universal Baireness implies property $(\ast)$ in $\mathsf{ZF}$.
We do not know whether the converse can be proved in $\mathsf{ZF}$. The proof uses the standard notion of the leftmost branch of a tree. That is, if $S$ is an illfounded tree on $\Ord$, the \emph{leftmost branch} of $S$ is the sequence $\lb(S) \in \Ord^{\omega}$ defined recursively by letting $\lb(S)(n)$ be the least ordinal $\alpha$ such that the tree \[\{ s \in S : s \subseteq (\lb(S)\restrict n)^{\frown}\langle \alpha \rangle \vee  (\lb(S)\restrict n)^{\frown}\langle \alpha \rangle \subseteq s\}\] is illfounded. Usually this operation is applied to a tree of the form $T_{x}$ (as above), where $x \in \omega^{\omega}$ and $T$ is a tree on $\omega \times \Ord$, to find a witness to the statement $x \in \p[T]$.

\begin{lem}\label{uBimpliesstar}
 Let $A \subseteq \omega^\omega$ be universally Baire, let $X$ be a topological space with a regular open base, and let $f \colon X \to \omega^\omega$ be a continuous function.  Then the preimage $f^{-1}[A]$ has the Baire property in $X$.
\end{lem}

\begin{proof}
 The proof is a minor modification of the proof given in \cite{FenMagWoo}, which uses Choice for some steps. Let $\operatorname{RO}(X)$ denote the collection of all regular open subsets of $X$, which is a complete boolean algebra with negation given by the complement of the closure, and suprema given by the interior of the closure of the union. Considering $\mathrm{RO}(X)$ as a poset (under inclusion, with the emptyset excluded), let $S$ and $T$ be trees witnessing that $A$ is $\mathrm{RO}(X)$-Baire, with $\p[S] = A$. For each pair $n,m \in \omega$, let $X_{n,m}$ be the set of $x \in X$ with $f(x)(n) = m$. Then each $X_{n,m}$ is in $\mathrm{RO}(X)$. Let $\dot{g}$ be the $\operatorname{RO}(X)$-name for an element of $\omega^{\omega}$ consisting of the pairs $(X_{n,m}, (n,m))$ for each $n, m \in \omega$, so that each $X_{n,m}$ forces that $\dot{g}(n) = m$. 
 
Densely many conditions in $\mathrm{RO}(X)$ decide whether the realization of $\dot{g}$  is in the projection of $S$ or $T$. Let $U$ be the union of the conditions forcing $\dot{g}$ into $\p[S]$, and let $V$ be the corresponding set for $T$. Since intersections of regular open sets are regular open, $U$ and $V$ are disjoint. Since $X$ has a regular open base, $U \cup V$ is dense. It suffices to see that $f^{-1}[A] \bigtriangleup U$ is meager in $X$. 
 
 
 For each $\sigma \in \omega^{<\omega}$, let $X_{\sigma}$ be $\bigcap_{n < |\sigma|}X_{n,\sigma(n)}$, i.e., the set of $x \in X$ for which $\sigma$ is an initial segment of $f(x)$. For each $n \in \omega$, let $\mathcal{D}_{n}$ be the set of conditions $Y$ of $\mathrm{RO}(X)$ which are contained in either $U$ or $V$ and also in some set of the form $X_{\sigma}$, for some $\sigma \in \omega^{n}$, and which decide the first $n$ elements of the leftmost branch of whichever of $S_{\dot{g}}$ or $T_{\dot{g}}$ is illfounded (i.e., $S_{\dot{g}}$ for sets contained in $U$ and $T_{\dot{g}}$ for sets contained in $V$). Let $\tau_{n}(Y)$ denote the length-$n$ initial segment of this leftmost branch as decided by $Y$. Note then that if $Y \in \mathcal{D}_{n}$ is contained in some such $X_{\sigma}$, then the pair $(\sigma,\tau_{n}(Y))$ is in the corresponding tree.  
 
 For each $n \in\omega$ let $D_{n}$ be the union of all the members of $\mathcal{D}_{n}$. Then each $D_{n}$ is dense open.  It suffices now to see that if $x$ is in $U \cap \bigcap_{n \in \omega}D_{n}$ then $f(x) \in A$, and if $x \in V \cap \bigcap_{n \in \omega}D_{n}$ then $f(x) \not\in A$. 
 
 To see the former claim, fix $x \in U$ and regular open sets $Y_{n} \in \mathcal{D}_{n}$ $(n \in \omega)$ with $x \in \bigcap_{n \in \omega} Y_{n}$. Each $Y_{n}$ is contained in some $X_{\sigma_{n}}$, in such a way that $f(x) = \bigcup_{n\in\omega}\sigma_{n}$. Furthermore, since the conditions $Y_{n}$ are compatible, the values $\tau_{n}(Y_{n})$ are compatible, and the pair $(f(x), \bigcup_{n \in \omega}\tau_{n}(Y_{n}))$ gives a branch through $S$, showing that $f(x) \in A$. The proof for the latter claim is the same. 
 \end{proof}
 
 

As a trivial consequence of the previous lemma, if $A$ is universally Baire then $A$ itself has the Baire property.
Note that if the Axiom of Choice holds then there is a set of reals that does not have the Baire property and therefore is not universally Baire.

%
%
%
%
%
%
%
%
% ...
%

The following standard lemma on $\mathbb{P}$-absolutely complementing pairs of trees is often useful.


\begin{lem}\label{lem-projection-inclusion}
 Let $T$ be a tree on $\omega \times \Ord$, let $\mathbb{P}$ be a poset, and let $(U, \tilde{U})$ be a $\mathbb{P}$-absolutely complementing pair of trees on $\omega \times \Ord$.
 If $\p[T] \subseteq \p[U]$, then $1_{\mathbb{P}} \forces \p[T] \subseteq \p[U]$.
\end{lem}

\begin{proof}
 If $\p[T] \subseteq \p[U]$, then $\p[T] \cap \p[\tilde{U}] = \emptyset$.
 This implies that \[1_{\mathbb{P}} \forces \p[T] \cap \p[\tilde{U}] = \emptyset,\] by the absoluteness of wellfoundedness of the tree \[\{(s,t,u) : (s,t) \in T \wedge (s,u) \in \tilde{U}\}\] on $\omega \times \Ord \times \Ord$.
 Because 
$(U, \tilde{U})$ is $\mathbb{P}$-absolutely complementing, it follows that 
$1_{\mathbb{P}} \forces \p[T] \subseteq \p[U]$.
\end{proof}

Applying Lemma \ref{lem-projection-inclusion} in both directions, we immediately obtain the following result, which shows that although there may not be a canonical way to choose trees witnessing universal Baireness of a given set of reals, any such pair of trees gives canonical information about how to expand the set of reals in generic extensions.

\begin{lem}\label{lem-unique-expansion}
 Let $\mathbb{P}$ be a poset, let $A$ be a $\mathbb{P}$-Baire set of reals, and let $(T,\tilde{T})$ and $(U, \tilde{U})$ be $\mathbb{P}$-absolutely complementing pairs of trees witnessing that $A$ is $\mathbb{P}$-Baire.
 Then $1_{\mathbb{P}} \forces \p[T] = \p[U]$.
\end{lem}

The following notation is therefore well-defined.

\begin{defn}
 For a poset $\mathbb{P}$, a $\mathbb{P}$-Baire set of reals $A$, and a $V$-generic filter $G \subseteq \mathbb{P}$,
 taking any $\mathbb{P}$-absolutely complementing pair of trees $(T,\tilde{T})$ witnessing that $A$ is $\mathbb{P}$-Baire,
 the \emph{canonical expansion of $A$ to $V[G]$} is 
 \[A^{V[G]} = \p[T]^{V[G]}.\]
\end{defn}

\begin{rem}\label{canexprem} The canonical expansion $A^{V[G]}$ depends only on the model $V[G]$ and not the generic filter $G$. 
\end{rem}

We will sometimes use the following local version of universal Baireness, following Steel \cite{SteDMT} and Larson \cite{LarStationaryTower}.

\begin{defn}\label{defn-lt-kappa-uB}
Let $\kappa$ be a cardinal.
\begin{itemize}
 \item A pair of trees $(T,\tilde{T})$ on $\omega \times \Ord$ is
 $\kappa$-\emph{absolutely complementing} if
 it is $\mathbb{P}$-absolutely complementing for every 
 (wellordered) poset $\mathbb{P}$ of cardinality less than $\kappa$.
 \item A set of reals $A$ is \emph{$\kappa$-universally Baire} if $A = \p[T]$ for some $\kappa$-absolutely complementing pair of trees $(T,\tilde{T})$ on $\omega \times \Ord$.
\end{itemize}
We write $\uB_{\kappa}$ for the pointclass $\{ A \subseteq \omega^\omega : A \text{ is $\kappa$-universally Baire}\}$. 
\end{defn}

The reader should be warned that because 
our definition says ``cardinality less than $\kappa$'' rather than ``cardinality less than or equal to $\kappa$,'' what we call $\kappa^+$-universally Baire is equivalent
in $\mathsf{ZFC}$ to what Feng, Magidor, and Woodin call $\kappa$-universally Baire \cite[Theorem 2.1]{FenMagWoo}.

One way to produce $\kappa$-absolutely complementing trees is the Martin--Solovay construction from a system of $\kappa$-complete measures (which we briefly review in Section \ref{amodelsec}; see also Section 1.3 of \cite{LarStationaryTower}).
Another way, assuming the Axiom of Choice, is to amalgamate $\mathbb{P}$-absolutely complementing trees for various posets $\mathbb{P}$.
This approach yields the following result, whose well-known proof we include here for the reader's convenience.

\begin{lem}
 Assume $\mathsf{ZFC}$, let $\kappa$ be a cardinal, and let $A$ be a set of reals.
 If $A$ is $\mathbb{P}$-Baire for every poset $\mathbb{P}$ of cardinality less than $\kappa$,
 then $A$ is $\kappa$-universally Baire.
\end{lem}

\begin{proof}
Let $\eta = 2^{\mathord{<}\kappa}$ and
let
$\langle \mathbb{P}_\alpha : \alpha < \eta\rangle $ enumerate the set of all posets on
cardinals less than $\kappa$.
For each $\alpha < \eta$, choose a $\mathbb{P}_\alpha$-absolutely complementing pair of trees $(T_\alpha, \tilde{T}_\alpha)$ witnessing that $A$ is $\mathbb{P}_\alpha$-Baire.
Consider the  tree $T$ on $\omega \times \Ord$ defined by
\[T = \{ (s, (\alpha^\frown t) \restrict \left|s\right|) : \alpha < \eta \text{ and } (s,t) \in T_\alpha \}.\]
This construction immediately implies that $\p[T] = \bigcup_{\alpha < \eta} \p[T_\alpha]$ in every outer model of $V$.
Similarly, letting 
\[\tilde{T} = \{ (s, (\alpha^\frown t) \restrict \left|s\right|) : \alpha < \eta \text{ and } (s,t) \in \tilde{T}_\alpha \},\]
$\p[\tilde{T}] = \bigcup_{\alpha < \eta} \p[\tilde{T}_\alpha]$ in every outer model of $V$.
In $V$ we have $\p[T] = A$ and $\p[\tilde{T}] = \omega^{\omega} \setminus A$. It follows immediately from the construction of $T$ and $\tilde{T}$ that the union of their projections is $\omega^{\omega}$ in any forcing extension by a poset of cardinality less than $\kappa$. The absoluteness of the wellfoundedness of the tree
\[\{ (s,t,u) : (s,t) \in T \wedge (s,u) \in \tilde{T}\}\] implies that the projections of $T$ and $\tilde{T}$ are still disjoint in any such extension. 
That is, the pair $(T, \tilde{T})$ is 
$\kappa$-absolutely complementing.
%
% ...
%
% ...
\end{proof}

It follows that, assuming the Axiom of Choice, if $\lambda$ is a limit cardinal and $A$ is $\kappa$-universally Baire for every $\kappa < \lambda$, then $A$ is $\lambda$-universally Baire.

\section{Semiscales}\label{semiscalesec}


In this section, we work in $\mathsf{ZF}$, except where noted otherwise. A \emph{norm} on a set of reals $A$ is an ordinal-valued function on $A$. A \emph{prewellordering} of $A$ is a binary relation on $A$ that is transitive, total, and whose corresponding strict relation is wellfounded. A norm $\varphi$ on a set $A$ induces the prewellording $\{ (x,y) : \varphi(x) \leq \varphi(y)\}$, and a prewellordering $\leq$ induces the norm sending $x \in A$ to its $\leq$-rank. 

\begin{defn}\label{convssdef}
 Let $A$ be a set of reals and let $\vec{\varphi} = \langle \varphi_n : n \in \omega\rangle$ be a sequence of norms on $A$.
\begin{enumerate}
 \item For a sequence $\langle x_i : i < \omega \rangle$ of reals in $A$ and a real $y$,  we say that $\langle x_i : i < \omega\rangle$ \emph{converges to $y$ via $\vec{\varphi}$} if $x$ converges to $y$ and for every $n < \omega$ the sequence $\langle \varphi_n(x_i) : i < \omega\rangle$ is eventually constant.

 \item\label{semiscaleitem} We say that $\vec{\varphi}$ is a \emph{semiscale} on $A$ if
 for every sequence $\langle x_i : i < \omega\rangle$ of reals in $A$ and every real $y$, if $\langle x_i : i < \omega\rangle$ converges to $y$ via $\vec{\varphi}$ then $y \in A$. 
\end{enumerate}
\end{defn}

\begin{rem}\label{weakscalerem} The condition that $\langle x_{i} : i < \omega\rangle$ converges to $y$ via $\vec{\varphi}$ in Part (\ref{semiscaleitem}) of Definition \ref{convssdef} can equivalently be replaced with the stronger condition that $x_{i} \restrict i = y \restrict i$ for all $i < \omega$, and $\varphi_{n}(x_{i}) = \varphi_{n}(x_{j})$ whenever $n \leq i < j$. This can be seen by thinning a counterexample satisfying the weaker condition to produce one satisfying the stronger one. 
\end{rem}

If $A$ is a Suslin set of reals then there is a semiscale on $A$.
More specifically,
given a tree $T$ on $\omega \times \Ord$ such that $p[T] = A$,
we define the corresponding \emph{leftmost branch semiscale} $\vec{\varphi}^T$ on $A$ by letting 
$\varphi^{T}_{n}(x) = \lb(T_{x})(n)$. 
That is, for each $n < \omega$,
$\varphi^T_n(x) = \xi_n$ where
$\langle \xi_0,\ldots,\xi_n \rangle \in \Ord^{n+1}$ is lexicographically least such that $(x,f) \in [T]$ for some $f\in \Ord^\omega$ extending $\langle \xi_0,\ldots,\xi_n\rangle$.
It is straightforward to verify that this construction gives a semiscale on $A$.

Conversely, if there is a semiscale on $A$ and $\mathsf{CC}_\mathbb{R}$ holds, then $A$ is Suslin ($\mathsf{CC}_{\bbR}$, or \emph{countable choice for reals}, is the restriction of the Axiom of Choice to countable sets of nonempty sets of real numbers; it holds in the models considered in this paper and in fact follows from $\AD$ by a theorem of Mycielski).
To see this, given a semiscale $\vec{\varphi}$ on $A$, define the \emph{tree of} $\vec{\varphi}$ by
\[T_{\vec{\varphi}} = \{ (x \restrict i, \langle \varphi_0(x),\ldots,\varphi_{i-1}(x)\rangle) : x \in A,\, i \in \omega\}.\]
Clearly, $A \subseteq \p[T_{\vec{\varphi}}]$.  Conversely, suppose that $y \in \p[T_{\vec{\varphi}}]$, and fix an $f \in \Ord^\omega$ such that
  $(y,f) \in [T_{\vec{\varphi}}]$.
 For each $i < \omega$ choose a real $x_i$ witnessing $(y \restrict i, f \restrict i) \in T_{\vec{\varphi}}$, so that $x_i \in A$,  $y \restrict i = x_i \restrict i$ and 
 $f \restrict i = \langle\varphi_0(x_i),\ldots,\varphi_{i-1}(x_i)\rangle$.
Then the sequence $\langle x_i : i < \omega\rangle$ converges to $y$ via $\vec{\varphi}$, so $y \in A$.
This shows that $A = \p[T_{\vec{\varphi}}]$.

It follows that if $\mathsf{CC}_\mathbb{R}$ holds, then the Suslin sets of reals are exactly the sets of reals that admit semiscales. We don't know if this can be proved in $\mathsf{ZF}$.
The following lemma can be used in some situations to show that $A = \p[T_{\vec{\varphi}}]$ without assuming $\mathsf{CC}_\mathbb{R}$.
The lemma is not useful for showing that $A$ is Suslin, because that is one of the hypotheses.  Rather, it is useful when we want to represent $A$ as the projection of a tree that is definable from a given semiscale $\vec{\varphi}$.

\begin{lem}
 Let $A$ be a set of reals and let $\vec{\varphi}$ be a semiscale on $A$.
 Assume that the set $A$ and the relation 
 \[E = \{(\bar{n}, x, y) : n < \omega \text{ and } x,y \in A \text{ and } \varphi_n(x) = \varphi_n(y)\}\]
 are both Suslin, where $\bar{n}$ denotes the constant function from $\omega$ to $\{n\}$.   Then $\p[T_{\vec{\varphi}}] = A$.
\end{lem}

\begin{proof}
Clearly, $A \subseteq \p[T_{\vec{\varphi}}]$ as before.
Conversely, let $y \in \p[T_{\vec{\varphi}}]$.
Fix $f \in \Ord^\omega$ such that $(y,f) \in [T_{\vec{\varphi}}]$.
 Take a tree $T_A$ on $\omega \times \Ord$ witnessing that $A$ is Suslin, and take a tree $T_E$ on $\omega \times \omega \times \omega \times \Ord$ witnessing that $E$ is Suslin.

For $N < \omega$, define a \emph{full $N$-witness} to be an object of the form
\[ ( \langle x_i : i < N\rangle, \langle g_i : i < N\rangle, \langle h_{n,i,j} : n \le i \le j < N\rangle)\]
such that
\begin{itemize}
 \item for all $i < N$, $x_i \in \omega^\omega$ witnesses that $(y \restrict i,f \restrict i) \in T_{\vec{\varphi}}$,
meaning that $x_i \in A$, $y \restrict i = x_i \restrict i$ and 
 $f \restrict i = \langle \varphi_0(x_i),\ldots,\varphi_{i-1}(x_i)\rangle$;

 \item for all $i < N$, $g_i \in \Ord^\omega$ witnesses that $x_i \in A$ in the sense that $(x_i,g_i) \in [T_A]$;

 \item for all $n \le i \le j < N$, $h_{n,i,j} \in \Ord^\omega$ witnesses $\varphi_n(x_i) = \varphi_n(x_j)$ in the sense that $(\bar{n}, x_i, x_j, h_{n,i,j}) \in [T_E]$.
\end{itemize}
We define a \emph{partial $N$-witness} to be an object 
that can be obtained by restricting some full $N$-witness as above to $N$ to produce the following:
\[ ( \langle x_i \restrict N : i < N\rangle, \langle g_i \restrict N: i < N\rangle, \langle h_{n,i,j} \restrict N: n \le i \le j < N\rangle).\]
Note that every full $N$-witness can be extended to a full $(N+1)$-witness
by choosing any $x_N \in \omega^\omega$ witnessing $(y \restrict N,f \restrict N) \in T_{\vec{\varphi}}$
and then choosing $g_N$ and $h_{n,i,N}$ appropriately for all $n$ and $i$
such that $n \le i \le N$.

It follows that every partial $N$-witness can be extended to a partial $(N+1)$-witness:
first extend it to a full $N$-witness, extend that to a full $(N+1)$-witness, and then restrict that to a partial $(N+1)$-witness.
By \emph{extending} a partial $N$-witness to a partial $(N+1)$-witness, we mean that its sequences of length $N$ are extended to sequences of length $N+1$, and 
new sequences of length $N+1$
of the form $x_N \restrict (N+1)$, $g_N \restrict (N+1)$, and $h_{n,i,N} \restrict (N+1)$ are added.

The advantage of partial witnesses (as compared with full witnesses) is that they are essentially finite sequences of ordinals, so the Axiom of Choice is not required to choose among them.
We may therefore define a sequence $\langle\sigma_N : N < \omega\rangle$ where for each $N < \omega$, $\sigma_N$ is an $N$-witness and $\sigma_{N+1}$ extends $\sigma_N$.
Taking $\bigcup_{N < \omega} \sigma_N$, by which we really mean taking the union in each coordinate separately, we obtain an object
\[ ( \langle x_i : i < \omega\rangle, \langle g_i : i < \omega\rangle, \langle h_{n,i,j} : n \le i \le j < \omega\rangle). \]
Then for all $i < N$ we have
$x_i \in A$ as witnessed by $g_i \in \Ord^\omega$.
Furthermore, the sequence
$\langle x_i : i < \omega\rangle$ converges to $y$ since
$y \restrict i = x_i \restrict i$ for all $i < \omega$.
We might not have 
$\varphi_n(x_i) = f(n)$ for all $n$ and $i$ such that $n \le i < \omega$ because this property of full witnesses might be lost by passing to partial witnesses and taking unions.
However, we do have $\varphi_n(x_i) = \varphi_n(x_j)$ for all $n$, $i$, and $j$ such that $n \le i \le j < \omega$, because this is witnessed by 
$h_{n,i,j} \in \Ord^\omega$.
Therefore the sequence $\langle x_i : i < \omega\rangle$ converges to $y$ via $\vec{\varphi}$.
Since $\vec{\varphi}$ is a semiscale on $A$ it follows that $y \in A$, as desired.
\end{proof}




Let $A$ be a set of reals and let $\vec{\varphi}$ be a semiscale on $A$.
The \emph{canonical code} for $\vec{\varphi}$
is the set $R$ consisting of those triples $(\bar{n},x,y) \in \omega^\omega \times A \times A$ for which
$\varphi_{n}(x) \le \varphi_{n}(y)$, where again $\bar{n}$ denotes the constant sequence $(n,n,n,\ldots)$.
When this holds we say that $R$ codes $\vec{\varphi}$.

For any ternary relation $R \subseteq \omega^\omega \times \omega^\omega \times \omega^\omega$ and any $n \in \omega$, define
\[ R_n = \{ (x,y) \in \omega^\omega \times \omega^\omega : (\bar{n}, x, y) \in R\}.\]
Note that a set  $R \subseteq \omega^\omega \times \omega^\omega \times \omega^\omega$ codes a semiscale if and only if
each $R_n$ is a prewellordering of $A$, and
for every sequence $\langle x_i : n \in \omega\rangle$ of points in $A$ converging to a point $y \in \omega^\omega$,
if for each $n$ in $\omega$ there is an $m \in \omega$ such that $x_{i}R_{n}x_{j}$ for all $i,j \in \omega \setminus m$, then $y$ is in $A$. 


The following result is immediate from the definitions.

\begin{lem}\label{restrictscaleslemma}
 Let $A \subseteq \omega^\omega$ and let $R \subseteq \omega^\omega \times \omega^\omega \times \omega^\omega$ code a semiscale on $A$.
 Let $\mathcal{M}$ be an inner model of $\mathsf{ZF}$.
 If $R \cap \mathcal{M} \in \mathcal{M}$ then $A \cap \mathcal{M} \in \mathcal{M}$ and
$\mathcal{M}$ satisfies ``$R \cap \mathcal{M}$ codes a semiscale on $A \cap M$.''
\end{lem}

The following result is proved by standard arguments, but since we don't know of a reference we include a proof here. Higher-level arguments for similar absoluteness results are given in Section \ref{sec-DM-at-limit-of-lt-lambda-strongs}. 


\begin{lem}\label{expandscaleslemma}
  Let $A \subseteq \omega^\omega$ and let $R \subseteq \omega^\omega \times \omega^\omega \times \omega^\omega$ code a semiscale on $A$.
  Let $\mathbb{P}$ be a poset and let $G\subseteq \mathbb{P}$ be a $V$-generic filter on $\mathbb{P}$.
  If $A$ and $R$ are $\mathbb{P}$-Baire, then $V[G]$ satisfies ``$R^{V[G]}$ codes a semiscale on $A^{V[G]}$.''
\end{lem}

\begin{proof}
 Let $(T_A, T_{\neg A})$ be a pair of trees on $\omega \times \Ord$ witnessing that $A$ is $\mathbb{P}$-Baire, and let
 $(T_R, T_{\neg R})$ be a pair of trees on $\omega \times \omega \times \omega \times \Ord$ witnessing that $R$ is $\mathbb{P}$-Baire.
 The desired conclusion will follow from the fact that the wellfoundedness of various trees defined from these four trees is absolute from $V$ to $V[G]$.

 First, we show that $R^{V[G]}_n$ is a prewellordering of $A^{V[G]}$ for each $n \in \omega$.
 Let $n \in \omega$.
 We have $R_n \subseteq A \times A$,
so the trees
 \[\{ (s, t, u, v)  : (\bar{n}, s, t, u)   \in T_{R} \text{ and }  (s, v) \in T_{\neg A} \}\] and \[\{ (s, t, u, v)  : (\bar{n}, s, t, u)   \in T_{R} \text{ and } (t, v) \in T_{\neg A} \}\]
 on $\omega \times \omega \times \Ord \times \Ord$
 are wellfounded, and therefore 
 $R^{V[G]}_n \subseteq A^{V[G]} \times A^{V[G]}$ (here $\bar{n}$ refers to the constant sequence of length $|s|$ with value $n$).

Since $R_n$ is transitive, the tree on $\omega \times \omega \times \omega \times \Ord \times \Ord \times \Ord$
consisting of all tuples $(s, t, u, v, w, x)$ such that 
\begin{itemize}
 \item $(\bar{n}, s, t, v) \in T_{R}$,
 \item $(\bar{n}, t, u, w) \in T_{R}$, and
 \item $(\bar{n}, s, u, x) \in T_{\neg R}$
\end{itemize}
is wellfounded. It follows that $R^{V[G]}_n$ is transitive.

Since $R_n$ is total on $A$, the tree
on $\omega \times \omega \times \Ord \times \Ord \times \Ord \times \Ord$
consisting of all tuples
$(s, t, u, v, w, x)$
such that 
\begin{itemize}
 \item $(s, u)   \in T_A$
 \item $(t, v)   \in T_A$
 \item $(\bar{n}, s, t, w) \in T_{\neg R}$, and
 \item $(\bar{n}, t, s, x) \in T_{\neg R}$
\end{itemize}
is wellfounded. It follows that $R^{V[G]}_n$ is total on $A^{V[G]}$.

For the remaining parts, which are the wellfoundedness of each $R^{V[G]}_n$ and the 
semiscale condition, it will help to introduce the following notation.
Fix a recursive bijection $\omega \times \omega \to \omega$ denoted by $(i, j) \mapsto \langle i, j \rangle$.
Given
$s \in \omega^{\less\omega}$ and integers $i,j$, define the finite sequence
$s_{i,j}$ to be $\langle s(\langle i, 0\rangle), \ldots, s(\langle i, j-1\rangle)\rangle$
if all necessary values of $s$ are defined, and undefined otherwise.

The strict part of $R_n$ is wellfounded, so $R_n$ has no infinite strictly decreasing sequences.
 Then the tree on $\omega \times \Ord \times \Ord$ consisting of all tuples $(s, t, u)$
 such that for all $i,j \in \omega$ such that $s_{i,j}$ (and hence also $t_{i,j}$ and $u_{i,j}$) and $s_{i+1,j}$ are defined,
\begin{itemize}
 \item $(\bar{n}, s_{i+1, j}, s_{i,j}, t_{i,j}) \in T_R$ and
 \item $(\bar{n}, s_{i,j}, s_{i+1,j}, u_{i,j}) \in T_{\neg R}$
\end{itemize}
 is wellfounded, so there is no strictly $R^{V[K]}_n$-decreasing infinite sequence in any generic extension $V[K]$ (we apply this fact with $K = (G,H)$ in the next paragraph).

Because we are not assuming the Axiom of Choice, the previous paragraph does not immediately show that
 $R^{V[G]}_n$ is wellfounded.  
 To show that it is, let $S$ be a nonempty subset of $A^{V[G]}$ in $V[G]$.
 Let $\mathbb{Q}$ be the poset of strictly $R^{V[G]}_n$-decreasing finite sequences
 of reals in $S$, ordered by reverse inclusion,
 and take a $V[G]$ generic filter $H \subseteq \mathbb{Q}$.
 Because the tree on $\omega \times \Ord \times \Ord$
 defined above is still wellfounded in $V[G][H]$,
 the sequence
 $\bigcup H$ cannot be infinite, so it has a last element, which must be
 an $R^{V[G]}_n$-minimal element of $S$.

Lastly, we prove the semiscale condition.
Because $R$ codes a semiscale
on $A$, the tree on $\omega \times \omega \times \Ord \times \Ord \times \Ord$ consisting of all tuples
$(c, s, t, u, v)$ such that
for all $n,i,j \in \omega$ such that $c(n) \le \min(i,j)$ and
$s_{i,j}$ (and hence also $t_{i,j}$ and $u_{i,j}$) and $s_{i+i,j}$
are defined, 
\begin{itemize}
 \item $(\bar{n}, s_{i,j}, s_{i+1,j}, t_{i,j}) \in T_R$,
 \item $(\bar{n}, s_{i+1,j}, s_{i,j}, u_{i,j}) \in T_R$,
 \item $s_{i,c(n)} = s_{i+1,c(n)}$, and
 \item $(s_{i,n}, v \restrict n) \in T_{\neg A}$
\end{itemize}
is wellfounded (here $c(n)$ indicates a point such that the reals $\bigcup_{j}s_{i,j}$ for $i > c(n)$ agree both in the $n$-th norm and also for their first $n$ digits; we could have done without $c$ by using Remark \ref{weakscalerem}).  This, along with the fact that $R^{V[G]}_n$ is a prewellordering of $A^{V[G]}$ for each $n \in \omega$, implies that $R^{V[G]}$ is a semiscale on $A^{V[G]}$ in $V[G]$.
\end{proof}

The following definition is used several times in the next section.

\begin{defn}
 A set $\Gamma \subseteq \powerset(\omega^\omega)$
has the \emph{semiscale property} if every member of $\Gamma$ has a semiscale which is coded by a continuous preimage of a member of $\Gamma$. 

\end{defn}

See Kechris and Moschovakis \cite{KecMosScales} or Moschovakis \cite{MosDST2009} for more on semiscales.

\section{Canonical models for universal Baireness}\label{canmodelsec}

In this section, we continue to work in $\mathsf{ZF}$.
We show that if $\Gamma \subseteq \uB$ and $\Gamma$ is selfdual and has the semiscale property, then there is an inclusion-minimal inner model containing $\mathbb{R}$ and $\Gamma$ and satisfying the statement that every member of $\Gamma$ is universally Baire. 
In the context of our main theorem this minimal model has no sets of reals other than those in $\Gamma$, and therefore satisfies the statement that every set of reals is universally Baire.

\begin{defn}
We let $F$ be the class of all quadruples $(A,\mathbb{P},p,\dot{x})$ such that
\begin{itemize}
 \item $A \in \uB$,
  $\mathbb{P}$ is a poset,
  $p \in \mathbb{P}$,
  $\dot{x}$ is a $\mathbb{P}$-name for a real, and
 \item $p$ forces $\dot{x}$ to be in the canonical expansion of $A$.
\end{itemize}
\end{defn}

For sets $\Gamma \subseteq \uB$, we define a local version of $F$ that tells us how to compute the canonical extensions of sets in $\Gamma$ only:
\[F \restrict \Gamma = \{ (A,\mathbb{P},p,\dot{x}) \in F : A \in \Gamma\}.\]


We will eventually (in Lemma \ref{Fabslem}) show that the desired minimal model is $L^{F \restrict \Gamma}(\Gamma, \mathbb{R})$,
by which we mean the model constructed from the transitive set \[V_\omega \cup \mathbb{R} \cup \Gamma \cup \{\Gamma\}\] relative to the predicate $F \restrict \Gamma$.
First we show that the sets in $\Gamma$ are universally Baire in this model when $\Gamma$ is selfdual and has the semiscale property. 
In the proof it will be important that the model knows how to expand not only the sets in $\Gamma$, but also the relations coding semiscales on these sets and their complements. 


%
%

\begin{thm}\label{FreflectsuBthrm}
 If $\Gamma \subseteq \uB$ is selfdual and has the semiscale property, then
 $L^{F \restrict \Gamma}(\Gamma, \mathbb{R}) \models \Gamma \subseteq \uB$.
\end{thm}



\begin{proof}
	Let $\mathcal{M}$ denote $L^{F \restrict \Gamma}(\mathbb{R}, \Gamma)$ and fix $A \in \Gamma$. 
	Let $\vec{\varphi}_+$ and $\vec{\varphi}_-$ be semiscales on $A$ and $\omega^\omega \setminus A$ respectively
	such that the corresponding codes
	\begin{align*}
		R_+ &= \{ (\bar{n}, x, y) \in \omega^{\omega} \times A \times A : \varphi_{+,n}(x) \le \varphi_{+,n}(y) \}\\
		\text{and}\\
		R_- &= \{ (\bar{n}, x, y) \in \omega^{\omega} \times (\omega^\omega \setminus A) \times (\omega^\omega \setminus A) : \varphi_{-,n}(x) \le \varphi_{-,n}(y) \}
	\end{align*}
	are in $\Gamma$. 
	
	Consider an arbitrary set $Z \in \mathcal{M}$.  By Remark \ref{collapserem}, it suffices to find a $\Col(\omega,Z)$-absolutely complementing pair of trees for $A$ in $\mathcal{M}$.
	To do this, we will fix a $V$-generic filter $H \subseteq \Col(\omega,Z)$, find a complementing pair of trees for $A$ in $\mathcal{M}[H]$, and then show that these trees exist in $\mathcal{M}$ and are as desired. 
	
	Fixing $H$, we have by Lemma \ref{expandscaleslemma} that
	the expanded relations $R_+^{V[H]}$ and $R_-^{V[H]}$ code semiscales on the expanded sets $A^{V[H]}$ and $(\omega^\omega \setminus A)^{V[H]}$ respectively.
	Letting $\rho_+$ and $\rho_-$ be the standard $\Col(\omega,Z)$-names for the canonical expansions of the relations $R_+$ and $R_-$ respectively, we see that the restrictions $\rho_+ \cap \mathcal{M}$ and $\rho_- \cap \mathcal{M}$ are in $\mathcal{M}$, as they are coded into the predicate $F \restrict \Gamma$. Therefore the restricted relations $R_+^{V[H]} \cap {\mathcal{M}[H]}$ and $R_-^{V[H]} \cap {\mathcal{M}[H]}$ are in the model $\mathcal{M}[H]$. These restricted relations code semiscales $\vec{\psi}_+,\vec{\psi}_- \in \mathcal{M}[H]$ on the sets $A^{V[H]} \cap {\mathcal{M}[H]}$ and $(\omega^\omega \setminus A)^{V[H]} \cap {\mathcal{M}[H]}$ respectively, 
	by Lemma \ref{restrictscaleslemma}. 
	In $\mathcal{M}[H]$, let $S_+$ and $S_-$ be respectively the trees of the semiscales $\vec{\psi}_{+}$ and $\vec{\psi}_{-}$. It suffices to see that these trees are in $\mathcal{M}$, and that they project to complements in any forcing extension of $\mathcal{M}$ by $\Col(\omega, Z)$. 
	
	These two claims are proved by the same symmetry argument, showing that $S_{+}$ and $S_{-}$ do not depend on the generic filter $H$. More precisely, suppose toward a contradition that there exist conditions $p, p' \in \Col(\omega, Z)$, with $p \in H$, which force contradictory statements about the membership of some sequence $\sigma$ (from $\mathcal{M}$) in either of the trees $S_{+}$ and $S_{-}$, noting that these trees are definable in the $\Col(\omega, Z)$-extension of $\mathcal{M}$ from the realizations of $\rho_{+} \cap \mathcal{M}$ and $\rho_{-} \cap \mathcal{M}$. By the symmetry of $\Col(\omega, Z)$, there exists a $V$-generic filter $H' \subseteq \Col(\omega, Z)$, with $p' \in H'$, such that $\bigcup H$ and $\bigcup H'$ differ by only finitely many coordinates, so that $V[H] = V[H']$ and $\mathcal{M}[H] = \mathcal{M}[H']$. 
	We have then (by Remark \ref{canexprem}) that $A^{V[H]} = A^{V[H']}$, $(\omega^{\omega} \setminus A)^{V[H]} = (\omega^{\omega} \setminus A)^{V[H']}$, $R_{+}^{V[H]} = R_{+}^{V[H']}$ and $R_{-}^{V[H]} = R_{-}^{V[H']}$, from which it follows that the corresponding versions of $S_{+}$ and $S_{-}$ in $\mathcal{M}[H]$ and $\mathcal{M}[H']$ are the same, giving a contradiction. The same argument shows that no condition $p'$ can force that $S_{+}$ and $S_{-}$ fail to project to complements in a $\Col(\omega, Z)$-extension of $\mathcal{M}$, since they do project to complements in $\mathcal{M}[H]$. 
	\end{proof}
	
	

Next, we show that the definition of $F \restrict \Gamma$ is absolute to 
any inner model in which all sets in $\Gamma$ are universally Baire. It follows that any such inner model must contain $L^{F \restrict \Gamma}(\Gamma, \mathbb{R})$. 

\begin{lem}\label{Fabslem}
 Let $\Gamma \subseteq \uB$.
 Let $\mathcal{M}$ be a model of $\mathsf{ZF}$ such that \[\mathbb{R} \cup \Gamma \cup \{\Gamma\} \cup \Ord \subseteq \mathcal{M}\] and $\mathcal{M} \models \Gamma \subseteq \uB$.
 Then $(F \restrict \Gamma)^\mathcal{M} = (F \restrict \Gamma) \cap \mathcal{M}$ and $L^{F \restrict \Gamma}(\Gamma, \mathbb{R}) \subseteq \mathcal{M}$.
\end{lem}

\begin{proof}
To show that $(F \restrict \Gamma)^\mathcal{M} = (F \restrict \Gamma) \cap \mathcal{M}$, fix $A \in \Gamma$ and let $\mathbb{P}$ be a poset in $\mathcal{M}$.
Let $p \in \mathbb{P}$ and let $\dot{x} \in \mathcal{M}$ be a $\mathbb{P}$-name for a real.
We claim that
\[ 
\mathcal{M} \models (A, \mathbb{P}, p, \dot{x}) \in F \restrict \Gamma
\iff 
(A, \mathbb{P}, p, \dot{x}) \in F \restrict \Gamma.\]
Take a pair of trees $(T, \tilde{T}) \in \mathcal{M}$ witnessing that $A$ is $\mathbb{P}$-Baire in $\mathcal{M}$ and take a pair of trees $(U,\tilde{U})$ witnessing that $A$ is $\mathbb{P}$-Baire in $V$.
Because $\p[T] = \p[U] = A$ and $\p[\tilde{T}] = \p[\tilde{U}] = \omega^\omega \setminus A$ in $V$ and the pair $(U,\tilde{U})$ is $\mathbb{P}$-absolutely complementing, Lemma \ref{lem-projection-inclusion} gives
\[ 1_{\mathbb{P}} \forces (\p[T] \subseteq \p[U] \text{ and } \p[\tilde{T}] \subseteq \p[\tilde{U}]).\]
We want to show that the following statements are equivalent, for a condition $p \in \mathbb{P}$ and a $\mathbb{P}$-name $\dot{x}$ in $\mathcal{M}$:
\begin{enumerate}
 \item $\mathcal{M} \models p \forces \dot{x} \in \p[T]$; 
 \item $p \forces \dot{x} \in \p[U]$.
\end{enumerate}
To see this, assume first that (1) is true and take a $V$-generic filter $G \subseteq \mathbb{P}$ containing $p$.  Because $G$ is also $\mathcal{M}$-generic, we have
\[\dot{x}_G \in \p[T]^{M[G]} \subseteq \p[T]^{V[G]} \subseteq \p[U]^{V[G]}.\]
This shows that (2) is true.

Now assume that (1) is false and take a condition $q \le p$ such that
\[\mathcal{M} \models q \forces \dot{x} \notin \p[T].\]
Take a filter $G \subseteq \mathbb{P}$ containing $q$ that is $V$-generic, and hence also $M$-generic.
Then we have
\begin{align*}
 \dot{x}_G &\in (\omega^\omega \setminus \p[T])^{\mathcal{M}[G]}\\
 &= \p[\tilde{T}]^{\mathcal{M}[G]}\\
 &\subseteq \p[\tilde{T}]^{V[G]}\\
 &\subseteq \p[\tilde{U}]^{V[G]}\\
 &= (\omega^\omega \setminus \p[U])^{V[G]}.
\end{align*}
This shows that (2) is false.

To see that $L^{F \restrict \Gamma}(\Gamma, \mathbb{R}) \subseteq \mathcal{M}$, consider the model
$(L^{F \restrict \Gamma}(\Gamma, \mathbb{R}))^\mathcal{M} = L^{(F \restrict \Gamma)^\mathcal{M}}(\Gamma, \mathbb{R})$.  It is contained in $\mathcal{M}$, and its construction agrees with that of the model $L^{F \restrict \Gamma}(\Gamma, \mathbb{R})$
at every stage because $(F \restrict \Gamma)^\mathcal{M} = (F \restrict \Gamma) \cap \mathcal{M}$.
\end{proof}




It follows that $L^{F \restrict \Gamma}(\Gamma, \mathbb{R})$ sees the definition of $F$ and therefore sees its own construction.


\begin{cor}\label{Mabscor}
 Let $\Gamma \subseteq \uB$.  If $\Gamma$ is selfdual and has the semiscale property, then
 the model $L^{F \restrict \Gamma}(\Gamma, \mathbb{R})$ satisfies the statement $V = L^{F \restrict \Gamma}(\Gamma, \mathbb{R})$.
\end{cor}


\section{A model of $\mathsf{ZF} + \mathsf{AD}^+ + {}$``every set of reals is universally Baire.''}\label{amodelsec}

In this section we present our main theorem and outline its proof. First we recall some standard notation from the proof of Woodin's Derived Model Theorem, as presented in Steel \cite{SteDMT}. 
We let $\mathrm{ScS}$ denote the pointclass consisting of those $A \subseteq \omega^{\omega}$ such that $A$ and $\omega^{\omega} \setminus A$ are both Suslin (i.e., such that $A$ is Suslin and co-Suslin). 

\begin{defn}
 Let $\lambda$ be a limit of inaccessible cardinals and let $G \subseteq \Col(\omega,\mathord{<}\lambda)$ be a $V$-generic filter.  We define:
 \begin{itemize}
  \item $\mathbb{R}^*_G$ to be $\bigcup_{\xi < \lambda} \mathbb{R}^{V[G \restrict \xi]}$.
  \item $\HC^*_G$ to be $\bigcup_{\xi < \lambda} \HC^{V[G \restrict \xi]}$, where $\HC$ denotes the collection of hereditarily countable sets.
  \item $\Hom^*_G$ to be the pointclass consisting of all sets of the form $\p[S] \cap \mathbb{R}^*_G$, for some $\lambda$-absolutely complementing pairs of trees $(S,T)$ appearing in a model $V[G \restrict \xi]$ with $\xi < \lambda$.
 \end{itemize}
\end{defn}

We define the symmetric extension
$V(\mathbb{R}^*_G)$ to be $\HOD_{V \cup \mathbb{R}^*_G \cup \{\mathbb{R}^*_G\}}^{V[G]}$. 
By a theorem of G\"{o}del proved in Section 5.2 of \cite{Schindler:SetTheory}, 
$V(\mathbb{R}^{*}_{G})$ is a model of $\ZF$. 
By the definition of the model, $\mathbb{R}^*_G \subseteq \mathbb{R}^{V(\mathbb{R}^*_G)}$. The reverse containment $\mathbb{R}^{*}_{G} \subseteq \mathbb{R}^{V(\mathbb{R}^{*}_{G})}$ follows from the homogeneity of $\Col(\omega, \less\lambda)$ (by a standard argument which appears for instance on page 307 of \cite{SteDMT}): each member $x$ of $\mathbb{R}^{V(\mathbb{R}^{*}_{G})}$ is definable in $V[G]$ from $\{\mathbb{R}^{*}_{G}\}$ and parameters existing in some model $V[G \restrict \xi]$ with $\xi < \lambda$, and the homogeneity of the remainder of $\Col(\omega, \less\lambda)$ after $\xi$ implies that $x$ exists already in $V[G \restrict \xi]$. 

A similar homogeneity argument shows that 
$\ScS^{V(\mathbb{R}^*_G)} \subseteq \Hom^*_G$, since any pair of trees witnessing membership in $\ScS^{V(\mathbb{R}^*_G)}$ would exist in some such $V[G \restrict \xi]$ (being definable in $V[G]$ from $\{\bbR^{*}_{G}\}$ and parameters in such a model), and would have to be $\lambda$-absolutely complementing in order to project to complements in $V(\mathbb{R}^{*}_{G})$. 
Finally, $\Hom^{*}_{G} \subseteq \ScS^{V(\bbR^{*}_{G})}$ since $V(\bbR^{*}_{G})$ contains each model $V[G \restrict \xi]$ $(\xi < \lambda)$. 


With these definitions in hand we can state the main theorem of this paper.

\begin{mainthm*}
 Let $\lambda$ be a limit of Woodin cardinals and a limit of strong cardinals, let $G \subseteq \Col(\omega,\mathord{<}\lambda)$ be a $V$-generic filter and let $\mathcal{M}$ be the model \[L^F(\mathbb{R}^*_G, \Hom^*_G)^{V(\mathbb{R}^*_G)},\]
 where $L^F(\mathbb{R}^*_G, \Hom^*_G)$ denotes the class of all sets constructible from the transitive set $V_\omega \cup \mathbb{R}^*_G \cup \Hom^*_G$
 relative to the predicate $F$.
 Then the following hold. 
 \begin{enumerate}
  \item\label{item-ad-plus} $\mathcal{M} \models \ADR$.

  \item\label{item-powerset-R-M-is-homstar} $\powerset(\mathbb{R})^\mathcal{M} = \Hom^*_G$

  \item\label{item-every-set-ub} $\mathcal{M} \models {}$``every set of reals is universally Baire''. 

  \item\label{item-satisfies-V-is-HOD-P-R} $\mathcal{M} \models V = \HOD_{\powerset(\mathbb{R})}$.

 \end{enumerate}

Moreover, $\mathcal{M}$ has the following minimality property: if $\mathcal{M}' \subseteq V(\bbR^{*}_{G})$ is a model of $\mathsf{ZF} + {}$``every set of reals is universally Baire'' such that \[\mathbb{R}^*_G \cup \Hom^*_G \cup \Ord \subseteq \mathcal{M}',\] then $\mathcal{M} \subseteq \mathcal{M}'$.

\end{mainthm*}


\begin{rem}\label{relrem} By results of Moschovakis and Woodin (see Theorem 0.3 of \cite{Larson:Extensions}), the theory $\ZF$ + $\AD$ + ``all sets of reals are Suslin" implies $\AD^{+}$ (whose definition we give below). So $\AD^{+}$ holds in the model $\mathcal{M}$ from our Main Theorem. 
By results of Martin and Woodin (Theorem 13.1 of \cite{Larson:Extensions}), the theory $\ZF$ + $\AD$ + ``all sets of reals are Suslin" implies $\ADR$. So part (1) of the main theorem follows from part (3) plus the fact that $\mathcal{M} \models \AD$. We will in fact prove that $\mathcal{M} \models \AD^{+}$ on the way to proving part (3) of the main theorem. 
\end{rem}

\begin{rem}\label{oldnewrem}
A theorem of Woodin from the 1980s (Theorem 7.1 of \cite{SteDMT}) says, using the notation above, that if $\lambda$ is a limit of Woodin cardinals then $L(\bbR^{*}_{G}, \Hom^{*}_{G}) \models \AD^{+}$ and $\Hom^{*}_{G}$ is $\ScS^{L(\bbR^{*}_{G}, \Hom^{*}_{G})}$. The model $L(\bbR^{*}_{G}, \Hom^{*}_{G})$ is sometimes referred to as the old derived model. This is to contrast it with the (larger) new derived model $D(V, \lambda, G)$, i.e., $L(\Gamma_{+}, \bbR^{*}_{G})$, where $\Gamma_{+}$ is the set of $B \subseteq \bbR^{*}_{G}$ in $V(\bbR^{*}_{G})$ 
for which $L(B, \bbR^{*}_{G}) \models \AD^{+}$. Theorem 31 of \cite{WooSEMI} says that $D(V, \lambda, G) \models \AD^{+}$, again under the assumption that $\lambda$ is a limit of Woodin cardinals. The hypotheses of our main theorem imply (via Theorem \ref{thm-DM-at-limit-of-lt-lambda-strongs}) that the old and new derived models are equivalent (see Remark \ref{oldnewequivrem}). 
\end{rem}



Before starting the proof of the main theorem we review some material on homogeneously Suslin sets of reals (which is presented in more detail in \cite{LarStationaryTower, SteDMT}). A \emph{tree of measures} (on an ordinal $\gamma$, the choice of which will usually not concern us) is a family $\langle\mu_s : s \in \omega^{\mathord{<}\omega}\rangle$ of measures such that each measure $\mu_s$ concentrates on the set $\gamma^{\left|s\right|}$ and such that $\mu_t$ projects to $\mu_s$ whenever $t$ extends $s$, meaning that for each $A \in \mu_{s}$ the set of $\sigma \in \gamma^{|t|}$ with $\sigma \restrict |s| \in A$ is in $\mu_{t}$. Given a tree of measures $\vec{\mu} = \langle \mu_s : s \in \omega^{\mathord{<}\omega}\rangle$, 
for every real $x$ we let $\vec{\mu}_x$ denote the tower (i.e., projecting sequence) of measures $\langle \mu_{x \restrict n} : n<\omega\rangle$. We define the set of reals $\textbf{S}_{\vec{\mu}}$ to be the set of $x \in \omega^\omega$ for which $\vec{\mu}_x$ is wellfounded, i.e., for which the direct limit of the $\mu_{x \restrict n}$-ultrapowers of $V$ (for $n \in \omega$) induced by the projection maps is wellfounded. 
Given a cardinal $\kappa$, a set of reals $A$ is said to be \emph{$\kappa$-homogeneous} if $A = \textbf{S}_{\vec{\mu}}$ for some tree $\vec{\mu}$ of $\kappa$-complete measures.

By an observation of Woodin (see Steel \cite[Proposition~2.5]{SteDMT}) every tree $\vec{\mu}$ of $\kappa$-complete measures on $\gamma$ is a $\kappa$-homogeneity system for some tree $T$ on $\omega \times \gamma$, meaning that each measure $\mu_s$ concentrates on the set $T \cap \gamma^{\left|s\right|}$ and for every real $x \in \p[T]$ the tower $\vec{\mu}_x$ is well-founded. For every real $x \notin \p[T]$ then the tower $\vec{\mu}_x$ must be ill-founded, so $S_{\vec{\mu}} = \p[T]$. So all $\kappa$-homogeneous sets of reals are Suslin, and they are sometimes called \emph{$\kappa$-homogeneously Suslin} to emphasize this property. In the weakest nontrivial case, when $\kappa = \omega_{1}$, we say that $A$ is \emph{homogeneously Suslin}. 

By the Levy-Solovay theorem, for every generic extension $V[g]$ of $V$ by a poset of size less than $\kappa$, every $\kappa$-complete $\mu \in V$ induces a corresponding measure $\hat{\mu}$ in $V[g]$ and the ultrapower of the ordinals by $\mu$ in $V$ agrees with the ultrapower of the ordinals by $\hat{\mu}$ in $V[g]$. (Henceforth, we will denote both measures by $\mu$ where it will not cause confusion.) So we can also define the set $\textbf{S}_{\vec{\mu}}$ in small generic extensions, and we have $(\textbf{S}_{\vec{\mu}})^{V[g]} \cap V = (\textbf{S}_{\vec{\mu}})^V$.

Given a tree $\vec{\mu}$ of measures, the Martin--Solovay tree $\ms(\vec{\mu})$ is the tree of attempts to build a real $x$ and a sequence of ordinals witnessing the ill-foundedness of the corresponding tower $\vec{\mu}_x$ (see \cite{SteDMT, LarStationaryTower} for a precise definition.) The Martin--Solovay tree has the key property that
\[\textbf{S}_{\vec{\mu}} = \omega^\omega \setminus \p[\ms(\vec{\mu})].\]
If the measures in $\vec{\mu}$ are $\kappa$-complete, then definition of the Martin--Solovay tree is absolute to generic extensions via posets of cardinality less than $\kappa$, so the equality $\textbf{S}_{\vec{\mu}} = \omega^\omega \setminus \p[\ms(\vec{\mu})]$ continues to hold in such extensions. It follows that for any tree $T$ carrying a homogeneity system $\vec{\mu}$ consisting of $\kappa$-complete measures, the pair $(T, \ms(\vec{\mu}))$ is $\kappa$-absolutely complementing. This implies that every $\kappa$-homogeneous set of reals $A = \textbf{S}_{\vec{\mu}}$ is $\kappa$-universally Baire, and that whenever $V[g]$ is a generic extension of $V$ by a poset of size less than $\kappa$ we have $A^{V[g]} = (\textbf{S}_{\vec{\mu}})^{V[g]}$. In other words, the canonical extension of $A$ given by $\kappa$-homogeneity agrees with that given by $\kappa$-universal Baireness. 

Another property of the Martin--Solovay tree that we will use is the fact that its definition from $\vec{\mu}$ is continuous: for every $n<\omega$ the subtree $\ms(\vec{\mu}) \restrict n$ consisting of the first $n$ levels of $\ms(\vec{\mu})$ is determined by finitely many measures from $\vec{\mu}$.

\begin{rem}\label{homubequivrem}
If $\kappa$ is a cardinal and $\delta_0$ and $\delta_1$ are Woodin cardinals with $\kappa < \delta_0 < \delta_1$, then every $\delta_1^+$-universally Baire set of reals is $\kappa$-homogeneously Suslin. This follows by combining a theorem of Woodin (see Larson \cite[Theorem 3.3.8]{LarStationaryTower} or Steel \cite[Theorem 4.1]{SteDMT})
with a theorem of Martin and Steel \cite{MarSteProjDet}.
In particular, if $\lambda$ is a limit of Woodin cardinals then every $\lambda$-universally Baire set of reals is $\mathord{<}\lambda$-homogenously Suslin (i.e., $\kappa$-homogeneous for all $\kappa < \lambda$; we write $\Hom_{\less\lambda}$ for the set of $\less\lambda$-homgeneously Suslin sets of reals). As outlined above, the reverse inclusion follows from the Martin--Solovay construction.
\end{rem}

We now proceed to give the proof of the main theorem, assuming three facts that will be proved in later sections. As in the proof of the Derived Model Theorem, part (\ref{item-ad-plus}) will be obtained as an immediate consequence of a $\Sigma^2_1$ reflection property of the model $\mathcal{M}$, as follows.


\begin{lem*}[\ref{lem-Sigma-2-1-reflection}]
 Let $\lambda$ be a limit of Woodin cardinals and a limit of strong cardinals,  let $G \subseteq \Col(\omega,\mathord{<}\lambda)$ be a $V$-generic filter and let 
 $\mathcal{M}$ be the model $L^F(\mathbb{R}^*_G, \Hom^*_G)^{V(\mathbb{R}^*_G)}$.
 For every sentence $\varphi$, if there is a set of reals $A\in \mathcal{M}$ such that $(\HC^*_G,\mathord{\in},A) \models \varphi$, then there is a set of reals $A \in \Hom_{\mathord{<}\lambda}^V$ such that $(\HC^{V},\mathord{\in},A) \models \varphi$.
\end{lem*}

Taking Lemma \ref{lem-Sigma-2-1-reflection}
for granted, we
immediately get that $\mathcal{M} \models \AD$, since $\mathord{<}\lambda$-homogeneously Suslin sets of reals are determined whenever $\lambda$ is greater than a measurable cardinal, a fact which  was essentially shown by Martin in the course of proving determinacy for $\bfPi^1_1$ sets of reals from a measurable cardinal (see Theorem 33.31 of \cite{Jec2002} or Exercise 32.2 of \cite{KanHigherInfinite}).

Lemma \ref{lem-Sigma-2-1-reflection} also implies that $\mathcal{M} \models \AD^{+}$. One way to see this is to note that  $\mathsf{AD}^+$ can be reformulated as a $\Pi^2_1$ statement that is true in the pointclass $\Hom_{\mathord{<}\lambda}^V$.
In more (standard) detail, the axiom $\AD^{+}$ is the conjunction of the following three statements: 
\begin{itemize}
	\item $\DC_{\bbR}$
	\item Every set of reals is $\infty$-Borel. 
	\item $\less\Theta$-Determinacy. 
\end{itemize}

The axiom $\DC_{\bbR}$ says that every tree on $\bbR$ without terminal nodes has an infinite branch. Using a pairing function on $\bbR$, we can represent such a tree $T$ with the set $A$ of reals coding sequences in $T$. Since homogeneously Suslin sets are Suslin, they can be uniformized. It follows that if $A$ is homogeneously Suslin then there is a function $f$ picking for each $x \in A$ a $y \in A$ coding a proper extension in $T$ of the node coded by $x$.  
Starting with any $x \in A$ and coding the sequence $\langle f^{i}(x) : i \in \omega\rangle$ with a real, we get that an infinite branch through $A$ exists in $(\HC, \in, A)$. This shows that $\DC_{\bbR}$ holds (in $V$) for all trees on $\bbR$ coded by homogeneously Suslin trees. Lemma \ref{lem-Sigma-2-1-reflection} then implies that it holds for all trees on $\bbR$ in $\mathcal{M}$. 

Similarly, $\AD$ implies that $\less\Theta$-Determinacy holds for games with Suslin, co-Suslin payoff (see Corollary 7.3 of \cite{Larson:Extensions}). Since $\less\lambda$-homogeneously Suslin sets are Suslin and co-Suslin (being $\less\lambda$-universally Baire by the remarks above), and the games referred to in the definition of $\less\Theta$-Determinacy can be coded by sets of reals, Lemma \ref{lem-Sigma-2-1-reflection} then implies that $\less\Theta$-Determinacy holds in $\mathcal{M}$. Finally, the arguments in Section 8.2 of \cite{Larson:Extensions} show that if a set $A \subseteq \omega^{\omega}$ is $\infty$-Borel, then $A$ has an $\infty$-Borel code coded by a set of reals which is ${\bf \Sigma}^{1}_{1}$ in a prewellordering which is Wadge below either $A$ or its complement, i.e., one coded in the structure $(\HC, \in, A)$. Every homogeneously Suslin set $A$ is Suslin, and therefore has an $\infty$-Borel code definable over $(\HC, \in A)$. 



The proof of the $\Sigma^2_1$ reflection property (i.e., Lemma \ref{lem-Sigma-2-1-reflection}) is the most technically demanding part of this paper. It will be given in Section \ref{sec-sigma-2-1-reflection} after some background information and preliminary work are presented in Sections \ref{sec-R-genericity-iterations} and \ref{sec-F-absoluteness} respectively. Our proof will be similar to that given by Steel \cite{SteSTFree} for the model $L(\mathbb{R}^*_G, \Hom^*_G)$. Roughly speaking, the proof can be adapted to the model $\mathcal{M}$ because the information added by the predicate $F$ is canonical.

Next we outline the proof of part (\ref{item-powerset-R-M-is-homstar}) of the main theorem. It suffices to prove the inclusion $\powerset(\mathbb{R})^\mathcal{M} \subseteq \Hom^*_G$, since the reverse inclusion holds by the definition of the model $\mathcal{M}$. Let $A$ be a set of reals in $\mathcal{M}$. Note that $\mathsf{AD}^+$ holds in the model $L(A,\mathbb{R}^*_G)$ since it holds in $\mathcal{M}$ and,
due to its equivalent $\Pi^2_1$ reformulation,
$\AD^{+}$ is downward absolute to transitive inner models with the same reals (Theorem 8.22 of \cite{Larson:Extensions}). 
Therefore $A$ is in the new derived model $D(V,\lambda,G)$ (which was defined in Remark \ref{oldnewrem}). 
Since $\lambda$ is a limit of $\mathord{<}\lambda$-strong cardinals, $D(V,\lambda,G)$ satisfies the statement that every set of reals is Suslin, by a theorem of Woodin which appears as Theorem \ref{thm-DM-at-limit-of-lt-lambda-strongs} below.  
Therefore $A \in \ScS^{D(V,\lambda,G)} \subseteq \ScS^{V(\mathbb{R}^*_G)} = \Hom^*_G$, so we have $A \in \Hom^*_G$.

To prove part (\ref{item-every-set-ub}) of the main theorem, we begin with the observation that every set in $\Hom^*_G$ is $\mathord{<}\Ord$-universally Baire in the symmetric model $V(\mathbb{R}^*_G)$. This follows from the assumption that $\lambda$ is a limit of strong cardinals by a standard argument: a pair of trees $(S,T)$ for a $\Hom^*_G$ set appears $V[G\restrict \xi]$ for some $\xi < \lambda$, where it is $\lambda$-absolutely complementing. Since $\lambda$ is a limit of strong cardinals in $V$ there is a strong cardinal $\kappa < \lambda$ in $V[G\restrict \xi]$. For any cardinal $\chi$ then there is an elementary embedding $j \colon V \to M$ with critical point $\kappa$ and $j(\kappa) > \chi$ and $V_{\chi} \subseteq M$. Then $(j(S), j(T))$ is a $\chi$-complementing pair in $V(\bbR^{*}_{G})$, and $\p[S] = \p[j(S)]$.  




	

It then follows that sets in $\Hom^*_G$ are  (fully) universally Baire in $V(\mathbb{R}^*_G)$, by the following lemma.

\begin{lem}\label{lem-ordinal-ac-equivalent}
 Let $\lambda$ be a cardinal and let $G \subseteq \Col(\omega,\mathord{<}\lambda)$ be a $V$-generic filter. Then in the symmetric model $V(\mathbb{R}^*_G)$, every $\mathord{<}\Ord$-universally Baire set of reals is universally Baire. Moreover this holds locally: for every set $Z$ there is a cardinal $\eta$ such that every $\eta^{+}$-absolutely complementing pair of trees on $\omega \times \Ord$ is $\Col(\omega,Z)$-absolutely complementing.
\end{lem}

\begin{proof}
 Fix $Z$, let $\eta = \max(\lambda, \left|Z\right|^{V[G]})$ and let $(S,T)$ be an $\eta^{+}$-absolutely complementing pair of trees in $V(\mathbb{R}^*_G)$. Then, as above, because $S$ and $T$ are subsets of the ground model, a standard homogeneity argument shows that $(S,T) \in V[G \restrict \xi]$ for some ordinal $\xi < \lambda$. The pair $(S,T)$ is $\eta^{+}$-absolutely complementing in the model $V[G \restrict \xi]$ also, because this property is downward absolute (since $\lambda \leq \eta$ the cardinal successor of $\eta$ is the same in all models under consideration).

Since $\eta \ge \lambda$, every generic extension of $V[G]$ by the poset $\Col(\omega,\eta)$ is also a generic extension of $V[G \restrict \xi]$ by the poset $\Col(\omega,\eta)$. It follows that the pair $(S,T)$ is $\eta^{+}$-absolutely complementing in $V[G]$. Since $\eta \ge \left|Z\right|^{V[G]}$, this implies that the pair $(S,T)$ is $\Col(\omega, Z)$-absolutely complementing in $V[G]$. This property is downward absolute, so the pair $(S,T)$ is $\Col(\omega, Z)$-absolutely complementing in $V(\mathbb{R}^*_G)$ as desired.
\end{proof}

To prove part (\ref{item-every-set-ub}) of the main theorem, it remains to show that the universal Baireness of $\Hom^*_G$ sets in $V(\mathbb{R}^*_G)$ is absorbed by the model $\mathcal{M}$ via the $F$ predicate. This would follow from Theorem \ref{FreflectsuBthrm} if we knew that each member of $\Hom^{*}_{G}$ carried a semiscale coded by a set in $\Hom^{*}_{G}$. This fact is proved in Section \ref{sec-DM-at-limit-of-lt-lambda-strongs}, giving the following lemma, finishing the proof of part (\ref{item-every-set-ub}), and part (1), by Remark \ref{relrem}.

\begin{lem*}[\ref{lem-absorption-of-uB}]
 Let $\lambda$ be a limit of Woodin cardinals and a limit of strong cardinals, and  let $G \subseteq \Col(\omega,\mathord{<}\lambda)$ be a $V$-generic filter.
 Then every set of reals in $\Hom^*_G$ is universally Baire in $L^F(\mathbb{R}^*_G, \Hom^*_G)^{V(\mathbb{R}^*_G)}$.
\end{lem*}


For part (\ref{item-satisfies-V-is-HOD-P-R}) of the main theorem, parts (\ref{item-powerset-R-M-is-homstar}) and (\ref{item-every-set-ub}), along with Theorem \ref{Fabslem} (with $\Gamma = \Hom^{*}_{G}$) imply that the restriction of $F$ to $\mathcal{M}$ is the same as the predicate $F$ defined in $\mathcal{M}$. As in Corollary \ref{Mabscor}, it follows that $\mathcal{M}$ sees its own construction (in particular $\mathcal{M}= L^{F}(\bbR, \mathcal{P}(\bbR))^{\mathcal{M}}$), so every element of $\mathcal{M}$ is ordinal-definable in $\mathcal{M}$ from a member of $\mathcal{P}(\bbR)^{\mathcal{M}}$. 


For the last part of the theorem (the minimality property of $\mathcal{M}$), note that if $\mathcal{M}' \subseteq V(\bbR^{*}_{G})$ contains $\Hom^{*}_{G}$ and satisfies the statement that all sets of reals are universally Baire (and thus Suslin), then, since $\Hom^{*}_{G} = \ScS^{V(\bbR^{*}_{G})}$, $\Hom^{*}_{G} = \mathcal{P}(\bbR)^{\mathcal{M}'}$, and in particular $\Hom^{*}_{G}$ is an element of $\mathcal{M}'$. It follows from Lemma \ref{Fabslem} then that $\mathcal{M} \subseteq \mathcal{M}'$. 


%
%
%

To finish the proof of the main theorem, then, it remains only to prove Lemmas \ref{lem-Sigma-2-1-reflection} and \ref{lem-absorption-of-uB}, and Theorem \ref{thm-DM-at-limit-of-lt-lambda-strongs}.

\section{$\mathbb{R}$-genericity iterations}\label{sec-R-genericity-iterations}


As in Steel's stationary-tower-free proof \cite{SteSTFree} of $\Sigma^2_1$ reflection for the model $L(\mathbb{R}^*_G, \Hom^*_G)$, our proof of $\Sigma^2_1$ reflection for the model $\mathcal{M}$ of the main theorem will make fundamental use 
of the genericity iterations of Neeman \cite{NeeOptimalII}, repeated $\omega$ many times to produce an $\mathbb{R}$-genericity iteration.
We refer the reader also to Neeman \cite{NeeHandbook} for a thorough account of the use of $\mathbb{R}$-genericity iterations to prove $\mathsf{AD}$ in $L(\mathbb{R}^*_G)$ by a similar reflection argument.

The definitions we give below are more or less standard; we include them here to set out the notation that we will use in the rest of the paper. 
In the following definition, ``sufficient fragment" can be taken to mean sufficiently large for the definition (e.g. Woodin cardinals and generic extensions) to make sense. In practice $P_{0}$ will be the transitive collapse of a suitably large rank initial segment of the universe. 

\begin{defn}\label{defn-R-gen-iter}
 Let $P_0$ be a model of a sufficient fragment of $\ZFC$
and let $\bar{\lambda}$ be a limit of Woodin cardinals of $P_0$.
 An \emph{$\mathbb{R}$-genericity iteration of $P_0$ at $\bar{\lambda}$} is a sequence
 \[\langle P_{j}, i_{jk}, x_{\ell}, \delta_{\ell}, g_{\ell} : j \leq k \leq \omega, \ell < \omega \rangle,\] existing in a generic extension of $V$ by $\Col(\omega,\mathbb{R})$, such that:
 \begin{enumerate}
  \item $P_j$ is in $V$, for each $j < \omega$,
  \item $i_{jk}$ is an elementary embedding from $P_j$ to $P_k$, for all $j \leq k \leq \omega$, 
  \item $i_{jk}$ is an element of $V$, whenever $j \leq k < \omega$,
  \item $x_\ell \in \bbR^{V}$, for all $\ell < \omega$,
  \item each $\delta_{\ell}$ is a Woodin cardinal below  $i_{0,\ell}(\bar{\lambda})$ in $P_{\ell}$,
  \item\label{genitem} for all $\ell < \omega$, $g_{\ell} \in V$ is a $P_{\ell+1}$-generic filter for the poset $\Col(\omega,i_{\ell,\ell+1}(\delta_\ell))$ such that $x_\ell \in P_{\ell+1}[g_{\ell}]$, 
  \item $i_{kp} \circ i_{jk} = i_{jp}$ whenever $j \le k \le p \le \omega$,
  \item $P_\omega$ and the maps $i_{j,\omega}$ ($j < \omega$) are obtained as direct limits from the maps $i_{jk}$ ($j \leq k < \omega$),
  \item $i_{j,j+1}(\delta_j) < \delta_{j+1}$ for all $j < \omega$, 
  \item\label{soitextem} for all $j < k <  \omega$, the map $i_{j+1,k+1}$ has critical point above $i_{j,j+1}(\delta_j)$,
  \item $\{x_j : x < \omega\} = \mathbb{R}^V$,
  \item $i_{0,\omega}(\bar{\lambda}) = \sup\{i_{j,\omega}(\delta_j) : j < \omega\}$ and
  \item\label{pomegasymmtem} there is a $P_\omega$-generic filter $g \subseteq \Col(\omega,\mathord{<}i_{0,\omega}(\bar{\lambda}))$ such that \[\mathbb{R}^V = \bigcup\{ \bbR \cap P_{\omega}[g \restrict \xi] : \xi < i_{0,\omega}(\bar{\lambda})\}.\]
  
 \end{enumerate}
\end{defn}

Parallel to the notation in Section \ref{amodelsec}, in the context of item (\ref{pomegasymmtem}) above we write
$P_{\omega}(\bbR^{V})$ for \[\HOD^{P_{\omega}[g]}_{P_{\omega} \cup \bbR^{V} \cup \{\bbR^{V}\}}.\] By item (\ref{soitextem}), each embedding $i_{j+1,k+1}$ extends to an elementary embedding from $P_{j+1}[g_j]$ to $P_{k+1}[g_j]$, which we call $i_{j+1,k+1}^*$.

\begin{rem}\label{gisomrem} Each partial order $\Col(\omega,i_{\ell,\ell+1}(\delta_\ell))$ as in item (\ref{genitem}) above is forcing-equivalent to the corresponding partial order \[\Col(\omega, (i_{\ell - 1, \ell + 1}(\delta_{\ell - 1}),i_{\ell, \ell+1}(\delta_{\ell})])\] in the case $\ell > 0$, and $\Col(\omega \leq i_{01}(\delta_{0}))$ in the case $\ell= 0$. In practice (i.e., in the proof of Lemma \ref{lem-absolutely-definable}), the generic filter $g$ in item (\ref{pomegasymmtem}) is the union of the 
images of the filters $g_{\ell}$ under isomorphisms witnessing these forcing-equivalences (via bookkeeping that we leave to the reader). 
\end{rem}

To ensure that the symmetric model $P_\omega(\mathbb{R}^V)$ resembles $V$ in some sense,
we will use $\mathbb{R}$-genericity iterations of countable hulls $P_0$ of rank initial segments of $V$ where the iteration maps factor into the uncollapse map of the hull. The notation below suppresses $\gamma$, which is implicitly assumed to be associated with the map $\pi_{0}$. 

\begin{defn}
 Let $\lambda$ be a limit of Woodin cardinals, let $\gamma > \lambda$ be an ordinal and let $\pi_0 \colon P_0 \to V_{\gamma}$ be an elementary embedding of a countable transitive set $P_0$ into $V$ such that $\lambda \in \ran(\pi_0)$.

 An $\mathbb{R}$-genericity iteration of $P_0$ at $\pi_0^{-1}(\lambda)$, as in Definition \ref{defn-R-gen-iter}, is \emph{$\pi_0$-realizable} if there are elementary embeddings $\pi_j : P_j \to V_{\gamma}$ for $j \le \omega$ that commute with the iteration maps, meaning that $\pi_k \circ i_{jk} = \pi_j$ for $j \le k \le \omega$, and such that $\pi_j \in V$ for all $j < \omega$.
\end{defn}

For simplicity we will sometimes abuse notation by referring to a $\pi_0$-realizable genericity iteration of $P_0$ at $\pi_0^{-1}(\lambda)$ as a $\pi_0$-realizable genericity iteration of $P_0$ at $\lambda$; this should not cause any confusion because $\lambda$ is a limit of Woodin cardinals and $\pi_0^{-1}(\lambda)$ is a countable ordinal.

\begin{rem}\label{treemovecorrem}
If $C$ is a $\lambda$-universally Baire set of reals, $\gamma > \lambda$ is a limit ordinal, $P_{0}$ is a countable transitive set and $\pi_0 \colon P_0 \to V_{\gamma}$ is an elementary embedding with $(C,\lambda) \in \ran(\pi_0)$, then for every $\pi_0$-realizable genericity iteration $P_0 \to P_\omega$ at $\lambda$ we have $C \in P_\omega(\mathbb{R}^V)$, and in particular $C$ is equal to $\pi_\omega^{-1}(C)^{P_\omega(\mathbb{R}^V)}$, the canonical expansion of $\pi_\omega^{-1}(C)$ in $P_\omega(\mathbb{R}^V)$. 
To see this, note that $\ran(\pi_0)$ contains a $\lambda$-absolutely complementing pair of trees $(S,T)$ for $C$. Then  $\p[\pi_\omega^{-1}(S)] \subseteq \p[S] = C$ and $\p[\pi_\omega^{-1}(T)] \subseteq \p[T]$, so the canonical expansion of $\pi_\omega^{-1}(C)$ in $P_\omega(\mathbb{R}^V)$ is equal to $C$.
\end{rem}



\begin{rem}
 The notion of $\mathbb{R}$-genericity iteration given by 
 Definition \ref{defn-R-gen-iter}  is weaker than the usual one because it does not require the elementary embeddings $i_{jk} \colon P_j \to P_k$ to be iteration maps.
 (We could have called them something like ``$\mathbb{R}$-genericity systems'' instead.)
 We do it this way so that the reader unfamiliar with iteration trees can take the following lemma (which is implicit in Steel \cite{SteSTFree}) as a black box; its proof is the only place where iteration trees will appear in this paper. 
\end{rem}
 
The conditions on $\gamma$ in the hypotheses of Lemma \ref{lem-absolutely-definable} are chosen to make $V_{\gamma}$ satisfy the theory in Definition 1.1 of \cite{MarSteIterationTrees}, with respect to $\lambda$. Recall from Remark \ref{homubequivrem} that $\Hom_{\less\lambda} = \uB_{\lambda}$ when $\lambda$ is a limit of Woodin cardinals. Remark \ref{treemovecorrem} shows that the expression $P_\omega(\mathbb{R}^V) \models \psi[C,x]$ in the statement of the lemma makes sense, i.e., that $C$ is in $P_{\omega}(\mathbb{R}^{V})$. 
 


\begin{lem}\label{lem-absolutely-definable}
 Let $\lambda$ be a limit of Woodin cardinals and let $C$ be a $\lambda$-universally Baire set of reals.
 Let $P_{0}$ be a countable transitive set and let $\gamma$ be a limit ordinal of cofinality greater than $\lambda$. Let $\pi_0 \colon P_0 \to V_{\gamma}$ be an elementary embedding with $(C,\lambda) \in \ran(\pi_0)$. Suppose that $A$ is a set of reals in $V$ and that there is a binary formula $\psi$ such that for every $\pi_0$-realizable $\mathbb{R}$-genericity iteration $P_0 \to P_\omega$ at $\lambda$ in $V^{\Col(\omega,\mathbb{R})}$ and every real $x \in V$ we have
 \[ x \in A \iff P_\omega(\mathbb{R}^V) \models \psi[C,x].\]
 Then $A \in \Hom_{\mathord{<}\lambda}$.
\end{lem}

\begin{proof}
 By the wellfoundedness of the Wadge hierarchy on homogeneously Suslin sets (see Theorem 3.3.5 of \cite{LarStationaryTower} and the discussion before it), there is minimal $\kappa < \lambda$ such that every $\kappa$-homogeneously Suslin set of reals is $\less\lambda$-homogeneously Suslin.  By the elementarity of $\pi$ we have $\kappa \in \ran(\pi_0)$ and $(S,T) \in \ran(\pi_0)$ for some $\lambda$-absolutely complementing pair of trees $(S,T)$ for $C$.
 Define $(\bar{\lambda},\bar{\kappa},\bar{S},\bar{T}) = \pi_0^{-1}(\lambda,\kappa,S,T)$.
 Let $\delta_0$ denote the least Woodin cardinal of $P_0$ above $\bar{\kappa}$.
 
 Let $W$ denote the set of all (reals coding) ordered pairs $(\mathcal{T},b)$ such that
 \begin{itemize}
  \item $\mathcal{T}$ is a $2^\omega$-closed iteration tree on $P_0$ of length $\omega$ (meaning that each extender produces a $2^{\omega}$-closed ultrapower when applied to the model from which it is chosen),
  \item $\mathcal{T}$ is above $\bar{\kappa}$ and based on $\delta_0$ (meaning that each extender chosen has critical point above $\bar{\kappa}$ and von Neumann rank below the corresponding image of $\delta_{0}$),
  \item $b$ is a cofinal branch of $\mathcal{T}$, and
  \item the model $\mathcal{M}_b^{\pi_0\mathcal{T}}$ is wellfounded, where $\pi_0\mathcal{T}$ is the lifted tree on $V_{\gamma}$.
 \end{itemize}
  
As shown by Martin and Steel \cite[p. 27]{MarSteIterationTrees}, the condition that the model $\mathcal{M}_b^{\pi_0\mathcal{T}}$ is wellfounded implies that the branch $b$ is $\pi_0$-realizable (meaning that the branch embedding $i^\mathcal{T}_b$ factors into $\pi_0$), which in turn implies that the model $\mathcal{M}_b^{\mathcal{T}}$ is wellfounded.
  
  By a  result of K.~Windszus (see Steel \cite[Lemma 1.1]{SteSTFree}) the set $W$ is $\kappa$-homogeneously Suslin, and therefore $\mathord{<}\lambda$-homogeneously Suslin by our choice of $\kappa$. (We added the condition that $\mathcal{T}$ is based on $\delta_0$, which is harmless and convenient.) Note that essentially all of the complexity of $W$ comes from condition that $\mathcal{M}_{b}^{\pi_0\mathcal{T}}$ is wellfounded; the other conditions are arithmetic.
  
  Next we show (following Steel \cite[Claim 2 on p.~7]{SteSTFree}) that
 the following statements are equivalent, for all $x \in \bbR$:
  \begin{enumerate}
   \item\label{item:x-in-A} $x \in A$.
   \item\label{item:exists-T-in-W}

    There exist a pair $(\mathcal{T},b) \in W$
    and an $\mathcal{M}^\mathcal{T}_b$-generic filter \[g \subseteq \Col(\omega,i^\mathcal{T}_b(\delta_0))\]
    such that 
   $x \in \mathcal{M}^\mathcal{T}_b[g]$ and $\psi\big[\p[i^\mathcal{T}_b(\bar{S})],x\big]$ holds in
   the symmetric extension of $\mathcal{M}^\mathcal{T}_b[g]$ at its limit $i^\mathcal{T}_b(\bar{\lambda})$ of Woodin cardinals. (It doesn't matter which symmetric extension, by the homogeneity of the Levy collapse forcing.)
 \end{enumerate}

The two directions of the equivalence of \eqref{item:x-in-A} and \eqref{item:exists-T-in-W} can be proved by constructing a suitable $\pi_{0}$-realizable $\mathbb{R}$-genericity iteration of $P_{0}$ at $\lambda$. We prove the forward direction first and then note the changes needed for the reverse direction.  

Fix $x \in \mathbb{R}$ (at the end of the proof we will assume in addition that $x$ is in $A$), and note that because $\delta_0$ is a Woodin cardinal of $P_0$,
  by Theorem 7.16 of Neeman \cite{NeeHandbook} there is a $2^\omega$-closed iteration tree
  $\mathcal{T}$ on $P_0$ that is above $\bar{\kappa}$, based on $\delta_0$, of length $\omega$,
  such that for every cofinal wellfounded branch $b$ of $\mathcal{T}$ there is a $\mathcal{M}^\mathcal{T}_b$-generic filter $g \subseteq \Col(\omega,i^\mathcal{T}_b(\delta_0))$ such that 
  $x \in \mathcal{M}^\mathcal{T}_b[g]$.
  By Martin and Steel \cite[Corollary 5.7]{MarSteIterationTrees},
  the $2^\omega$-closed iteration tree $\pi_0\mathcal{T}$ on $V_{\gamma}$
  has a cofinal wellfounded branch $b$.
  For such a branch $b$, the pair $(\mathcal{T},b)$ is in $W$,
  and there is a $\mathcal{M}^\mathcal{T}_b$-generic filter $g \subseteq \Col(\omega,i^\mathcal{T}_b(\delta_0))$ such that 
  $x \in \mathcal{M}^\mathcal{T}_b[g]$.

  Now let $x_0 = x$ and let $\mathcal{T}_0$, $b_0$, and $g_0$ be such that $(\mathcal{T}_0,b_0) \in W$, 
  $g_0 \subseteq \Col(\omega,i^\mathcal{T}_b(\delta_0))$ is an $\mathcal{M}^{\mathcal{T}_0}_{b_0}$-generic filter, and
  $x \in \mathcal{M}^{\mathcal{T}_0}_{b_0}[g_0]$. (The previous paragraph shows that such $\mathcal{T}_0$, $b_0$, and $g_0$ exist.)
  Let $P_1 = \mathcal{M}^{\mathcal{T}_0}_{b_0}$ and let $i_{0,1} = i^{\mathcal{T}_{0}}_{b_{0}}$ be the branch embedding from $P_{0}$ to $P_{1}$. 
  As noted above, the condition that $\mathcal{M}_{b_0}^{\pi_0\mathcal{T}_0}$ is wellfounded
  implies that the branch $b_0$ is $\pi_0$-realizable, meaning that there is an elementary embedding $\pi_1 : P_1 \to V_{\gamma}$ such that $\pi_1 \circ i_{0,1} = \pi_0$.

  Now with the model $P_1$ in place of the model $P_0$, the map $\pi_1$ in place of the map $\pi_0$,
  a Woodin cardinal $\delta_1$ of $P_1$ such that $i_{0,1}(\delta_0) < \delta_1 < i_{0,1}(\bar{\lambda})$ in place of $\delta_0$, and an arbitrary real $x_1$, we can repeat the process above, applying the Neeman and Martin--Steel theorems again
  to get an iteration tree $\mathcal{T}_1$ on $P_1$ of length $\omega$,
  above $i_{0,1}(\delta_0)$ and based on $\delta_1$, a $\pi_1$-realizable branch
  $b_1$ of $\mathcal{T}_1$, and an $\mathcal{M}^{\mathcal{T}_1}_{b_1}$-generic filter $g_1 \subseteq \Col(\omega, i^{\mathcal{T}_1}_{b_1}(\delta_1))$ such that 
  $x_1 \in \mathcal{M}^{\mathcal{T}_1}_{b_1}[g_1]$.
  
We can repeat this process for $\omega$ many stages.
(Note that, for the purposes of showing the equivalence of (1) and (2), at a stage $j>0$, it suffices to use a $\pi_j$-realizable branch of $\mathcal{T}_j$ given by Martin and Steel \cite[Theorem 3.12]{MarSteIterationTrees}; the requirement that $\mathcal{M}_{b_j}^{\pi_j\mathcal{T}_j}$ is wellfounded was only necessary for stage $j=0$ to establish that $(\mathcal{T}_{0}, b_{0})\in W$.)
  Furthermore, given a generic enumeration $\{x_j : j < \omega\}$ of $\mathbb{R}^V$ in $V^{\Col(\omega,\mathbb{R})}$, we do this in such a way to produce a 
  $\pi_0$-realizable $\mathbb{R}$-genericity iteration $P_0 \to P_1 \to \cdots P_\omega$.
  With suitable bookkeeping, one can do this in such a way that the Woodin cardinals $i_{j,\omega}(\delta_j)$ for $j<\omega$ are cofinal in $i_{0,\omega}(\bar{\lambda})$, and the generic filters
  $g_j$ for $j<\omega$ fit together into a $P_\omega$-generic
  filter $g \subseteq \Col(\omega,\mathord{<}i_{0,\omega}(\bar{\lambda}))$ (as in Remark \ref{gisomrem}) such that $\mathbb{R}^V = \bigcup\{\mathbb{R} \cap P_{\omega}[g \restrict \xi] : \xi < i_{0\omega}(\bar{\lambda})\}$.
   
To see that \eqref{item:x-in-A} implies \eqref{item:exists-T-in-W},
   it remains to observe that the condition $x \in A$ is equivalent to the condition $P_\omega(\mathbb{R}^V) \models \psi\big[i_{0,\omega}(\bar{S}),x\big]$ by the hypothesis of the lemma and the fact that $C = \p[i_{0,\omega}(\bar{S})]^{P_\omega(\mathbb{R}^V)}$, as in Remark \ref{treemovecorrem}. 
   This in turn is equivalent to the condition that $\psi\big[\p[i_{0,1}(\bar{S})],x\big]$ holds in the symmetric extension of $P_1[g_0]$ at its limit $i_{0,1}(\bar{\lambda})$ of Woodin cardinals (i.e., that the first step of our iteration witnessed \eqref{item:exists-T-in-W}) by the fact that the iteration map $i_{1,\omega} : P_1 \to P_\omega$ extends to an elementary embedding $i^*_{1,\omega} : P_1[g_0] \to P_\omega[g_0]$.
   
The proof that \eqref{item:exists-T-in-W} implies \eqref{item:x-in-A} is the same, except that the first step of the genericity iteration is given by assuming \eqref{item:exists-T-in-W}, and the two equivalences in the previous paragraph are applied in reverse. 
   
  This characterization of the set $A$ given by the equivalence of \eqref{item:x-in-A} and \eqref{item:exists-T-in-W} shows that $A$ is a projection of a $\mathord{<}\lambda$-homogeneously Suslin set (essentially all of whose complexity comes from $W$), i.e., that $A$ is $\mathord{<}\lambda$-weakly homogeneously Suslin. It follows that $\lambda$ is $\mathord{<}\lambda$-homogeneously Suslin by the main theorem from the proof of projective determinacy (Martin and Steel \cite[Theorem~5.11]{MarSteProjDet}; see also Theorem 3.3.13 of \cite{LarStationaryTower}).
\end{proof}

\color{black}


%

\section{Absoluteness of the $F$ predicate}\label{sec-F-absoluteness}

Although our main theorem is concerned only with the predicate $F$ as it is defined in the symmetric model $V(\mathbb{R}^*_G)$, the proof of $\Sigma^2_1$ reflection will need to consider the predicate as it is defined in other models (in particular $V$ and $\mathbb{R}$-genericity iterates of countable hulls of rank initial segments of $V$) and show that the definition satisfies a certain absoluteness property between these models (as in Lemma \ref{lem-DM-of-P-omega} below).

First we will need a lemma that shows that, in the presence of Woodin cardinals, the universally Baire sets added generically over a countable model that embeds into $V$ are closely related to universally Baire sets in $V$. 


In the presence of Woodin cardinals, the pointclass of (sufficiently) homogeneously Suslin sets is closed under complementation in a strong sense. That is, if $\delta$ is a Woodin cardinal, $Y$ is a set of $\delta^+$-complete measures with $|Y| < \delta$, and $\kappa < \delta$ is a cardinal, then by Steel \cite[Lemma~2.1]{SteSTFree} there is a ``tower-flipping'' function $f$ that associates to every finite tower $\langle\rho_0,\ldots,\rho_{n-1}\rangle$ of measures from $Y$ a length-$n$ tower $f(\langle\rho_0,\ldots,\rho_{n-1}\rangle)$ of $\kappa$-complete measures with the following properties. First, $f$ respects extensions of finite towers, so that to every infinite tower $\vec{\rho}$ of measures from $Y$ it continuously associates an infinite tower $\bigcup_{n<\omega} f(\vec{\rho} \restrict n)$ of $\kappa$-complete measures. Second, this associated tower $\bigcup_{n<\omega} f(\vec{\rho} \restrict n)$ is well-founded if and only if the given tower $\vec{\rho}$ is ill-founded.  Third, this property of the tower-flipping function continues to 
hold in all generic extensions by posets of size less than $\kappa$.

The following lemma shows that in the presence of a Woodin cardinal, the Martin--Solovay construction can be used to produce absolutely complementing trees of a nice form, which we will use in Lemma \ref{lem-containment-and-absoluteness} to prove an absoluteness result for the $F$ predicate. A similar idea appears in the proof of Steel \cite[Theorem~2.2, Subclaim~1.1]{SteSTFree}.

 For a tree $T$ and a natural number $n$ we let $T \mathord{\restrict} n$ be the subset of $T$ consisting of all nodes of length less than $n$.

\begin{lem}\label{lem-map-ub-sets-up-no-strong}
 Let $P$ be a countable transitive set, let $\gamma$ be a limit ordinal and let $\pi \colon P \to V_{\gamma}$ be an elementary embedding. Let $\bar{\alpha}$ be a cardinal of $P$ and suppose that $\bar{\delta}$ is a Woodin cardinal of $P$ with $\bar{\alpha} < \bar{\delta}$. Let $g \subseteq \Col(\omega,\bar{\alpha})$ be a $P$-generic filter in $V$ and let $A$ be a $\bar{\delta}^+$-homogeneously Suslin set of reals in $P[g]$.

 Then for each cardinal $\bar{\kappa}$ of $P$ with $\bar{\alpha} < \bar{\kappa} < \bar{\delta}$ there is a $\bar{\kappa}$-absolutely complementing pair of trees $(\bar{S},\bar{T}) \in P[g]$ for $A$ such that, for every $n < \omega$, the restrictions $\bar{S}\restrict n$ and $\bar{T} \restrict n$ to finite levels are in $P$, and, letting $S = \bigcup_{n<\omega} \pi(\bar{S}\restrict n)$ and $T = \bigcup_{n<\omega} \pi(\bar{T}\restrict n)$, the pair $(S,T) \in V_{\gamma}$ is $\pi(\bar{\kappa})$-absolutely complementing.
\end{lem}

\begin{proof}
 Let $A = \textbf{S}_{\vec{\mu}}$ where $\vec{\mu}$ is a tree of $\bar{\delta}^+$-complete measures in $P[g]$.  By the Levy--Solovay theorem each measure $\mu_s$ is induced by the $\bar{\delta}^+$-complete measure $\mu_s \cap P \in P$, which we will also denote by $\mu_s$. In $P$ there is a set $Y$ consisting of $\bar{\delta}^+$-complete measures such that $|Y|^{P} = \bar{\alpha}$  every member  of $\vec{\mu}$ is induced by a member of $Y$.
 

 Given a cardinal $\bar{\kappa}$ of $P$ with $\bar{\alpha} < \bar{\kappa} < \bar{\delta}$ (as in the statement of the lemma), let $f \in P$ be a tower-flipping function that continuously associates to every tower of measures from $Y$ a tower of $\bar{\kappa}$-complete measures. In $P[g]$, define the tree of $\bar{\kappa}$-complete measures $\vec{\nu} = \langle\nu_s : s \in \omega^{\mathord{<}\omega}\rangle$ that is continuously associated to $\vec{\mu}$ by $f$ in the sense that $\langle\nu_{s \restrict n} : n \le \left|s\right|\rangle = f(\langle\mu_{s \restrict n} :n \le \left|s\right|\rangle)$ for every finite sequence $s \in \omega^{\mathord{<}\omega}$. So in $P[g]$ we have $\textbf{S}_{\vec{\nu}} = \omega^\omega \setminus\textbf{S}_{\vec{\mu}}$ by the tower-flipping property of $f$. Moreover, $\pi(f)$ is a tower-flipping function in $V_{\gamma}$ by the elementarity of $\pi$, so we have $\textbf{S}_{\pi``\vec{\nu}} = \omega^\omega \setminus\textbf{S}_{\pi``\vec{\mu}}$ in $V_{\gamma}$.

 
 In $P[g]$ define the Martin--Solovay trees $\bar{S} = \ms(\vec{\mu})$ and $\bar{T} = \ms(\vec{\nu})$, and in $V_{\gamma}$ define the Martin--Solovay trees $S = \ms(\pi``\vec{\mu})$ and $T = \ms(\pi``\vec{\nu})$. Because for each $n \in \omega$ the first $n$ levels of the Martin--Solovay tree are determined by finitely many measures (in the ground model, by the absoluteness of the construction of the Martin-Solovay tree between $P$ and $P[g]$ for $\bar{\delta}^{+}$-complete measures) we have $(\bar{S} \restrict n,\bar{T} \restrict n) \in P$ for all $n<\omega$. For the same reason, we have $S = \bigcup_{n<\omega} \pi(\bar{S}\restrict n)$ and $T = \bigcup_{n<\omega} \pi(\bar{T}\restrict n)$.

In the model $P[g]$ the pair $(\bar{S},\bar{T})$ is $\bar{\kappa}$-absolutely complementing: this follows from the fact that $\p[\bar{S}] = \omega^\omega \setminus \p[\bar{T}]$ (in $P[g]$) and the fact that $\bar{S}$ and $\bar{T}$, being Martin--Solovay trees, are each $\bar{\kappa}$-absolutely complemented.
 A similar argument shows that in $V$ the pair $(S,T)$ is $\pi(\bar{\kappa})$-absolutely complementing, as desired.
\end{proof}

Using a strong cardinal, we can strengthen the previous lemma to give any desired degree of absolute complementation.

\begin{lem}\label{lem-map-ub-sets-up}
 Let $P$ be a countable transitive set, let $\gamma$ be a limit ordinal and let $\pi \colon P \to V_{\gamma}$ be an elementary embedding. Let $\bar{\alpha}$ be a cardinal of $P$ and let $\bar{\delta}$ be a Woodin cardinal of $P$ with $\bar{\alpha} < \bar{\delta}$. Let $g \subseteq \Col(\omega,\bar{\alpha})$ be a $P$-generic filter in $V$ and let $A$ be a $\bar{\delta}^+$-homogeneously Suslin set of reals in $P[g]$.

 Let $\bar{\kappa}$ and $\bar{\eta}$ be cardinals of $P$ with $\bar{\alpha} < \bar{\kappa} < \bar{\delta} < \bar{\eta}$ such that $\bar{\kappa}$ is $(\bar{\eta}+1)$-strong in $P$. Then there is an $\bar{\eta}$-absolutely complementing pair of trees $(\bar{S},\bar{T}) \in P[g]$ for $A$ such that for every $n < \omega$ the restrictions $\bar{S}\restrict n$ and $\bar{T} \restrict n$ to finite levels are in $P$, and, letting $S = \bigcup_{n<\omega} \pi(\bar{S}\restrict n)$ and $T = \bigcup_{n<\omega} \pi(\bar{T}\restrict n)$, the pair $(S,T) \in V_{\gamma}$ is $\pi(\bar{\eta})$-absolutely complementing.
\end{lem}

\begin{proof}
 Let $(\kappa, \eta) = \pi(\bar{\kappa}, \bar{\eta})$. By Lemma \ref{lem-map-ub-sets-up-no-strong} there is a $\bar{\kappa}$-absolutely complementing pair of trees $(\bar{S}_0, \bar{T}_0)$ in $P[g]$ for $A$ such that for every $n < \omega$ the restrictions $\bar{S}_0\restrict n$ and $\bar{T}_0 \restrict n$ to finite levels are in $P$, and, letting $S_0 = \bigcup_{n<\omega} \pi(\bar{S}_0\restrict n)$ and $T_0 = \bigcup_{n<\omega} \pi(\bar{T}_0\restrict n)$, the pair $(S_0,T_0) \in V_{\gamma}$ is $\kappa$-absolutely complementing.
  
 In $P$, let $\bar{E}$ be an extender witnessing that $\bar{\kappa}$ is $(\bar{\eta}+1)$-strong. The ultrapower map $j_{\bar{E}} \colon P \to \Ult(P, \bar{E})$ extends canonically to a map $P[g] \to \Ult(P[g], \bar{E})$, which we will also denote by $j_{\bar{E}}$. In $V_{\gamma}$ the extender $E = \pi(\bar{E})$ witnesses that $\kappa$ is $(\eta+1)$-strong, and we have an ultrapower map $j_E : V_{\gamma} \to \Ult(V_{\gamma}, E)$ with $j_{E}(\kappa) > \eta$ and $V_{\eta + 1} \subseteq \Ult(V_{\gamma}, E)$. 

 Letting $\bar{S} = j_{\bar{E}}(\bar{S}_0)$ and $\bar{T} = j_{\bar{E}}(\bar{T}_0)$, the pair of trees $(\bar{S}, \bar{T})$ is an $\bar{\eta}$-absolutely complementing pair for $A$ in $P[g]$. Similarly, letting $S = j_E(S_0)$ and $T = j_E(T_0)$, the pair of trees $(S, T) \in V_{\gamma}$ is $\eta$-absolutely complementing. Because $\pi \circ j_{\bar{E}} = j_E \circ \pi$, we have $S = \bigcup_{n<\omega} \pi(\bar{S}\restrict n)$ and $T = \bigcup_{n<\omega} \pi(\bar{T}\restrict n)$ as desired.
\end{proof}

Applying Lemma \ref{lem-map-ub-sets-up} to $\mathbb{R}$-genericity iterations yields the following result.
Part (\ref{item-containment}) below can also be derived as an immediate consequence of
Steel \cite[Subclaim~1.1]{SteSTFree}, which applies more generally to any limit $\lambda$ of Woodin cardinals.  However, in our situation we can use the fact that $\lambda$ is also a limit of strong cardinals to give a proof of part (\ref{item-containment}) using Lemma \ref{lem-map-ub-sets-up} instead.  Part (\ref{item-absoluteness}) then follows from a similar argument.



\begin{lem}\label{lem-containment-and-absoluteness}
 Let  $\lambda$ be a limit of Woodin cardinals and a limit of strong cardinals and let $\gamma$ be a limit ordinal of cofinality greater than $\lambda$. 
Let $P_{0}$ be a countable transitive set and let $\pi_0 \colon P_0 \to V_{\gamma}$ be an elementary embedding with $\lambda \in \ran(\pi_0)$. 
Let $P_\omega$ be obtained from $P_0$ by a $\pi_0$-realizable $\mathbb{R}$-genericity iteration at $\lambda$ and let $A$ be a set of reals in $\uB^{P_\omega(\mathbb{R}^V)}$. Then the following hold. 
\begin{enumerate}
 \item \label{item-containment}
  $A \in \uB^V$
 \item \label{item-absoluteness}
  For every set $Z \in P_\omega(\mathbb{R}^V) \cap V$, every condition $p \in \Col(\omega,Z)$, and every $\Col(\omega,Z)$-name $\dot{x} \in P_\omega(\mathbb{R}^V) \cap V$ for a real, the statement ``$p$ forces $\dot{x}$ to be in the canonical expansion of $A$'' is absolute between $P_\omega(\mathbb{R}^V)$ and $V$.
\end{enumerate}
\end{lem}

\begin{proof}
In terms of the notation for $\mathbb{R}$-genericity iterations given in (and just after) Definition \ref{defn-R-gen-iter}, the set $A \in \uB^{P_\omega(\mathbb{R}^V)}$ is represented by a pair of absolutely $\pi^{-1}_{\omega}(\lambda)$-complementing trees in $P_{\omega}(\bbR^{V})$, which, by the usual homogeneity argument, exist in $P_{\omega}[g_{j_{0}}]$, for some $j_{0} < \omega$.  Since $P_{\omega}$ is the direct limit of $\langle P_{i} : i < \omega \rangle$, these trees can be taken to be in the range of $i^{*}_{j,\omega}$ (the canonical extension of $i_{j,\omega}$ to $P_{j}[g_{j-1}]$) for some positive $j < \omega$.  
Otherwise stated, $A$ comes from a set $A_j \in \uB^{P_j[g_{j-1}]}$ appearing at stage $j$, in the sense that $A = i_{j,\omega}^*(A_j)^{P_\omega(\mathbb{R}^V)}$. 
 

To prove part (\ref{item-containment}), we apply Lemma \ref{lem-map-ub-sets-up} with $P = P_j$, 
$\gamma = \gamma$, $\pi = \pi_j$, $g = g_{j-1}$, $A = A_j$, $\bar{\alpha} = i_{j-1,j}(\delta_{j-1})$, $\bar{\kappa}$ equal to the least strong cardinal of $P_j$ above $i_{j-1,j}(\delta_{j-1})$, $\bar{\delta}$ equal to the least Woodin cardinal of $P_j$ above $\bar{\kappa}$, and $\bar{\eta} = \pi_j^{-1}(\lambda)$. This gives us a pair of $\pi^{-1}_j(\lambda)$-absolutely complementing trees $(\bar{S},\bar{T}) \in P_j[g_{j-1}]$ for $A_j$ such that for every $n < \omega$ the restrictions $\bar{S}\restrict n$ and $\bar{T} \restrict n$ to finite levels are in $P_j$, and, letting $S = \bigcup_{n<\omega} \pi_j(\bar{S}\restrict n)$ and $T = \bigcup_{n<\omega} \pi_j(\bar{T}\restrict n)$, the pair $(S,T)$ is $\lambda$-absolutely complementing in $V$.
 
Since $\pi_{\omega} \circ i_{j,\omega} = \pi_{j}$, $\pi_{\omega}(i_{j,\omega}(\bar{S} \restrict n)) = \pi_{j}(\bar{S} \restrict n)$ and $\pi_{\omega}(i_{j,\omega}(\bar{T} \restrict n)) = \pi_{j}(\bar{T} \restrict n)$ for each $n \in \omega$, so $S = \pi_{\omega}(i^{*}_{j,\omega}(\bar{S}))$ and $T = \pi_{\omega}(i^{*}_{j,\omega}(\bar{T}))$. 
This implies that $\pi_\omega `` (i_{j,\omega}^*(\bar{S})) \subseteq S$ and $\pi_\omega `` (i_{j,\omega}^*(\bar{T})) \subseteq T$.
It follows that the realization map $\pi_\omega$ takes branches of $i_{j,\omega}^*(\bar{S})$ and $i_{j,\omega}^*(\bar{T})$ pointwise to branches of $S$ and $T$ respectively, so the projections of these trees satisfy the inclusions $\p[i_{j,\omega}^*(\bar{S})] \subseteq \p[S]$ and $\p[i_{j,\omega}^*(\bar{T})] \subseteq \p[T]$.

Because $A = \p[i_{j,\omega}^*(\bar{S})] = \omega^\omega \setminus \p[i_{j,\omega}^*(\bar{T})]$ in $P_\omega(\mathbb{R}^V)$, these two inclusions imply that $A = \p[S] = \omega^\omega \setminus \p[T]$ in $V$. Therefore $A$ is in $\uB_\lambda^V$, which is equal to $\uB^V$ because there is a strong cardinal less than $\lambda$.

For part (\ref{item-absoluteness}), note that our genericity iteration was formed in some generic extension of $V_{\gamma}$ by the poset $\Col(\omega,\mathbb{R})$, and the resulting model $P_\omega(\mathbb{R}^V)$ is countable there.  In particular the set $Z$ is countable there. So in a generic extension of $V_{\gamma}$ by the poset $\Col(\omega,\mathbb{R}) \times \Col(\omega,\omega)$, which is forcing-equivalent to $\Col(\omega,\mathbb{R})$, we can find a filter $H \subseteq \Col(\omega,Z)$ that contains the condition $p$ and is both $V$-generic and $P_\omega(\mathbb{R}^V)$-generic. Let $x = \dot{x}_H$. We want to see that the two ways of expanding the set $A$, given by $A \in \uB^{P_\omega(\mathbb{R}^V)}$ and $A \in \uB^V$ respectively, agree on whether $x$ is a member. This would follow from Lemma \ref{lem-projection-inclusion} if $P_\omega(\mathbb{R}^V) \in V$, but it seems that the condition $P_\omega(\mathbb{R}^V) \in V$ may fail in general, which is why we seem to need Lemma \ref{lem-map-ub-sets-up}.

By Lemma \ref{lem-ordinal-ac-equivalent}, in the model $P_\omega(\mathbb{R}^V)$ there is an ordinal $\eta_\omega \ge \pi_\omega^{-1}(\lambda)$ such that every $\eta_\omega$-absolutely complementing set of trees is $Z$-absolutely complementing. Increasing $j$ (from the first paragraph of the proof of the lemma) if necessary, we may assume that $\eta_\omega = i_{j,\omega}(\bar{\eta})$ for some ordinal $\bar{\eta} \in P_j$.

Now we can apply Lemma \ref{lem-map-ub-sets-up} as in the proof of
part (\ref{item-containment}), except that now $\bar{\eta}$ may be strictly larger than $\pi_j^{-1}(\lambda)$, to get a pair of $\bar{\eta}$-absolutely complementing trees $(\bar{S},\bar{T}) \in P_j[g_{j-1}]$ for $A_j$ such that for every $n < \omega$ the restrictions $\bar{S}\restrict n$ and $\bar{T} \restrict n$ to finite levels are in $P_j$, and, letting $S = \bigcup_{n<\omega} \pi_j(\bar{S}\restrict n)$ and $T = \bigcup_{n<\omega} \pi_j(\bar{T}\restrict n)$ in $V_{\gamma}$, the pair $(S,T)$ is $\eta$-absolutely complementing where $\eta = \pi_j(\bar{\eta})$. In particular the pair $(S,T)$ is $\mathbb{R}$-absolutely complementing. Moreover the pair $(S_{\omega}, T_{\omega}) = ( i_{j,\omega}^*(\bar{S}), i_{j,\omega}^*(\bar{T}))$ is $\eta_\omega$-absolutely complementing in $P_\omega[g_{j-1}]$, so also in $P_\omega(\mathbb{R}^V)$. It follows from our choice of $\eta_\omega$ that $(S_{\omega}, T_{\omega})$ is $Z$-absolutely complementing in $P_\omega(\mathbb{R}^V)$.

We can use the pair of trees $(S_{\omega}, T_{\omega})$
to decide membership of $x$ in the canonical expansion of $A$ from the point of view of $P_\omega(\mathbb{R}^V)$, and use the pair of trees $(S,T)$ to decide membership of $x$ in the canonical expansion of $A$ from the point of view of $V$. As in the proof of part (\ref{item-containment}), we get the same answer in both cases, because the realization map $\pi_\omega$ takes branches of 
$S_{\omega}$
pointwise to branches of $S$, and takes branches of 
$T_{\omega}$
pointwise to branches of $T$.
\end{proof}

The first part of Lemma \ref{lem-containment-and-absoluteness} implies (assuming its hypotheses) that $\uB^{P_{\omega}(\bbR^{V})}$ is a Wadge-initial segment of $\uB^{V}$. In conjunction with the second part, this implies that 
\[ L^F(\mathbb{R}, \uB)^{P_\omega(\mathbb{R}^V)} \subseteq L^F(\mathbb{R}, \uB)^V.\]
Lemma \ref{lem-DM-of-P-omega} below is a stronger statement that implies this inclusion, phrased in terms of the following definition.


\begin{defn} Given an ordinal $\alpha$, we let $\uB \restrict \alpha$ denote the collection of universally Baire sets of Wadge rank less than $\alpha$,  and let $F_\alpha$ be the class of all quadruples $(A,\mathbb{P},p,\dot{x})$ in $F$ for which $A \in \uB \restrict \alpha$.  
\end{defn}


If $\alpha$ is at least the Wadge rank of the pointclass $\uB$, then we simply have $\uB \restrict \alpha = \uB$. 
Using these definitions we can state an immediate consequence of Lemma \ref{lem-containment-and-absoluteness}:

\begin{lem}\label{lem-DM-of-P-omega}
 Let $\lambda$ be a limit of Woodin cardinals and a limit of strong cardinals, and let $\gamma$ be a limit ordinal of cofinality greater than $\lambda$.  Let $P_{0}$ be a countable transitive set and let $\pi_0 \colon P_0 \to V_{\gamma}$ be an elementary embedding with $\lambda \in \ran(\pi_0)$. Let $P_\omega$ be obtained from $P_0$ by a $\pi_0$-realizable $\mathbb{R}$-genericity iteration at $\lambda$. Then we have
 \[ L^F(\mathbb{R}, \uB)^{P_\omega(\mathbb{R}^V)} = L_\beta^{F_\alpha}(\mathbb{R}, \uB\restrict \alpha)^V,\]
 where $\alpha$ is the Wadge rank of the pointclass $\uB^{P_\omega(\mathbb{R}^V)}$ and $\beta$ is the ordinal height of the model $P_\omega$.
\end{lem}

\section{$\Sigma^2_1$ reflection}\label{sec-sigma-2-1-reflection}

In this section we prove the following $\Sigma^2_1$ reflection result, following Steel's stationary-tower-free proof \cite{SteSTFree} of $\Sigma^2_1$ reflection for the model $L(\mathbb{R}^*_G, \Hom^*_G)$.

\begin{lem}\label{lem-Sigma-2-1-reflection}
 Let $\lambda$ be a limit of Woodin cardinals and a limit of strong cardinals, let $G \subseteq \Col(\omega,\mathord{<}\lambda)$ be a $V$-generic filter, and let $\mathcal{M}$ be the model $L^F(\mathbb{R}^*_G, \Hom^*_G)^{V(\mathbb{R}^*_G)}$. For every sentence $\varphi$, if there is a set of reals $A\in \mathcal{M}$ such that $(\HC^*_G,\mathord{\in},A) \models \varphi$, then there is a set of reals $A \in \Hom_{\mathord{<}\lambda}^V$ such that $(\HC^{V},\mathord{\in},A) \models \varphi$.
\end{lem}

Recall that $\Hom^*_G = \uB^{V(\mathbb{R}^*_G)}$, as $\lambda$ is a limit of strong cardinals, so we have an alternative characterization of $\mathcal{M}$ that will be useful in this section:
\[\mathcal{M} = L^F(\mathbb{R}^*_G, \uB)^{V(\mathbb{R}^*_G)}.\]

Supposing that the model $\mathcal{M}$ has a $\varphi$-witness, the idea of the proof is to take a countable hull of some $V_{\gamma}$ containing a $\varphi$-witness, and then to do an $\mathbb{R}$-genericity iteration of the hull to get a $\varphi$-witness that is a subset of $\mathbb{R}^V$. More specifically, we start by fixing a limit ordinal $\gamma$ of cofinality greater than $\lambda$, so that every set of reals in $\mathcal{M}$ is in $L_{\gamma}^{F}(\bbR^{*}_{G}, \Hom^{*}_{G})^{V_{\gamma}(\bbR^{*}_G)}$. Considering an $\bbR$-genericity iteration of the transitive collapse $P_{0}$ of an elementary substructure of $V_{\gamma}$, the elementarity of the corresponding map $\pi_\omega \colon P_\omega \to V_{\gamma}$ (and the definability of forcing) gives a $\varphi$-witness in the model $L^F(\mathbb{R},\uB)^{P_\omega(\mathbb{R}^V)}$, which is equal to the model $L_\beta^{F_\alpha}(\mathbb{R}, \uB\restrict \alpha)^V$ by Lemma \ref{lem-DM-of-P-omega}, where $\alpha$ is the Wadge rank of the pointclass $\uB^{P_\omega(\mathbb{R}^V)}$ and $\beta$ is the ordinal height of the model $P_\omega$.  In particular, there is a $\varphi$-witness in the model $L^F(\mathbb{R},\uB)^V$.

We will show that the least $\varphi$-witness arising in this manner is in $\Hom_{\mathord{<}\lambda}$. To do this, we will consider different countable hulls $P_0$ and different genericity iterations of these hulls giving rise to possibly different models of the form $L^F(\mathbb{R}, \uB)^{P_\omega(\mathbb{R}^V)}$. We will see that the least $\varphi$-witness will be present in all of these models and will admit a uniform definition there in the sense of  Lemma \ref{lem-absolutely-definable}.

We consider two cases involving two different notions of least $\varphi$-witness. The cases are exhaustive but not necessarily  mutually exclusive; if they overlap, either one can be used.

\begin{case}\label{case-all-of-uB}
For cofinally many elementary countable hulls $\pi_0 \colon P_0 \to V_{\gamma}$ with $\lambda \in \ran(\pi_0)$, every $\pi_0$-realizable $\mathbb{R}$-genericity iteration $P_0\to P_\omega$ at $\lambda$ has the property that $\uB^{P_\omega(\mathbb{R}^V)} = \uB^V$.
\end{case}

In this case, let $\beta$ be the least ordinal such that the model $L_\beta^F(\mathbb{R}, \uB)^V$ has a $\varphi$-witness (which is less than $\gamma$, as noted above). Take a set of reals $C \in \uB^V$ such that this model 
contains a $\varphi$-witness that is ordinal-definable from the parameter $C$ and the predicate $F$, and let $A \in L_\beta^F(\mathbb{R}, \uB)^V$ be the least such $\varphi$-witness in the canonical well-ordering of sets ordinal definable over $L^{F}_{\beta}(\mathbb{R}, \uB)^{V}$ from $C$ and $F$.

Take a countable hull $\pi_0 \colon P_0 \to V_{\gamma}$ such that $(C, \lambda) \in \ran(\pi_0)$ and for every $\pi_0$-realizable $\mathbb{R}$-genericity iteration $P_0\to P_\omega$ at $\lambda$ we have $\uB^{P_\omega(\mathbb{R}^V)} = \uB^V$. For every such $\mathbb{R}$-genericity iteration we have $\beta \in P_\omega$ by the minimality of $\beta$; otherwise the model $L^F(\mathbb{R}, \uB)^{P_\omega(\mathbb{R}^V)}$ would not be tall enough to reach a $\varphi$-witness (by Lemma \ref{lem-DM-of-P-omega}), violating the elementarity of $\pi_\omega$. Moreover, the set $C$ is in $P_\omega(\mathbb{R}^V)$ because it is the canonical expansion of the set $\pi_\omega^{-1}(C) \in P_\omega$, as in Remark \ref{treemovecorrem}. 

For every $\pi_0$-realizable $\mathbb{R}$-genericity iteration $P_0\to P_\omega$ at $\lambda$, the ordinal $\beta$ is definable in $P_\omega(\mathbb{R}^V)$ as the least ordinal such that the model $L_\beta^F(\mathbb{R}, \uB)^{P_\omega(\mathbb{R}^V)}$ contains a $\varphi$-witness, and in turn the least $\varphi$-witness $A$ is definable from $C$ in the model $L_\beta^F(\mathbb{R}, \uB)^{P_\omega(\mathbb{R}^V)}$ by the same definition as in the model $L_\beta^F(\mathbb{R}, \uB)^V$, which is the same model. This implies that $A \in \Hom_{\mathord{<}\lambda}^V$ by Lemma \ref{lem-absolutely-definable} as desired.

\begin{case}\label{case-part-of-uB}
 For cofinally many sufficiently elementary countable hulls $\pi_0 \colon P_0 \to V_{\gamma}$ with $\lambda \in \ran(\pi_0)$, there is a $\pi_0$-realizable $\mathbb{R}$-genericity iteration $P_0\to P_\omega$ at $\lambda$ such that $\uB^{P_\omega(\mathbb{R}^V)} \subsetneq \uB^V$.
\end{case}

In this case we will need the following claim, which says that we can put an upper bound on the pointclass necessary to construct a $\varphi$-witness.

\begin{claim*}
 Under the assumption of Case \ref{case-part-of-uB}, there is a
 set of reals $D \in \uB^V$ such that, letting $D^*$ denote the canonical expansion $D^{V(\mathbb{R}^*_G)}$ and letting $\alpha^*$ denote the Wadge rank of $D^*$ in $V(\mathbb{R}^*_G)$, the model
\[L^{F_{\alpha^*}}(\mathbb{R}^*_G, \uB \restrict \alpha^*)^{V(\mathbb{R}^*_G)}\]
 contains a $\varphi$-witness.
\end{claim*}

\begin{proof}
 Take an elementary countable hull $\pi_0 \colon P_0 \to V_{\gamma}$ with $\lambda \in \ran(\pi_0)$ and take a $\pi_0$-realizable $\mathbb{R}$-genericity iteration $P_0\to P_\omega$ at $\lambda$ such that $\uB^{P_\omega(\mathbb{R}^V)} \subsetneq \uB^V$. Then by Lemma \ref{lem-DM-of-P-omega} and the elementarity of the realization map $\pi_\omega \colon P_\omega \to V_{\gamma}$ there is a $\varphi$-witness in the model $L^{F_\alpha}_\beta(\mathbb{R}, \uB \restrict \alpha)^V$, where $\alpha$ is the Wadge rank of the pointclass $\uB^{P_\omega(\mathbb{R}^V)}$ and $\beta = \Ord^{P_\omega}$.

 Take a set of reals $D \in \uB^V$ of Wadge rank $\alpha$, take a sufficiently elementary countable hull $\pi_0' \colon P_0' \to V_{\gamma}$ with $(D, \lambda) \in \ran(\pi_0')$, and take a $\pi_0'$-realizable $\mathbb{R}$-genericity iteration $P_0'\to P_\omega'$ at $\lambda$ such that $\uB^{P_\omega'(\mathbb{R}^V)} \subsetneq \uB^V$.

 The model $P_\omega'(\mathbb{R}^V)$ sees the set $D$ as the canonical expansion of $\pi_\omega^{-1}(D)$ and it sees the pointclass $\uB^V \restrict \alpha$ as its collection of $\uB$ sets of Wadge rank less than that of $D$. If we have $\Ord^{P_\omega} \le \Ord^{P_\omega'}$ then the model $P_\omega'(\mathbb{R}^V)$ is tall enough to see that a $\varphi$-witness can be built from $\uB^V \restrict \alpha$ using the $F$ predicate, and the claim follows from the elementarity of the realization map $\pi_\omega' \colon P_\omega' \to V_{\gamma}$.

 On the other hand, if $\Ord^{P_\omega} > \Ord^{P_\omega'}$ then we can repeat this procedure with the hull $\pi_0' \colon P_0' \to V_{\gamma}$ in place of the hull $\pi_0 \colon P_0 \to V_{\gamma}$.  (This is why we made sure to take the second hull with $\uB^{P_\omega'(\mathbb{R}^V)} \subsetneq \uB^V$ also.) Repeating this procedure for $\omega$ many steps,
 either at some step we get a set $D \in \uB^V$ satisfying the claim, or else we get an infinite decreasing sequence of ordinals $\Ord^{P_\omega} > \Ord^{P_\omega'} > \Ord^{P_\omega''} > \cdots$, a contradiction.
\end{proof}

Now we can proceed as in Case \ref{case-all-of-uB}.  Let $D \in \uB^V$ be a set as in the claim. Taking any countable hull $\pi_0 : P_0 \to V_{\gamma}$ with $D \in \ran(\pi_0)$ and taking any genericity iterate $P_0 \to P_\omega$ of this hull, by the elementarity of the factor map $\pi_\omega$ and the proof of Lemma \ref{lem-DM-of-P-omega} we see that there is a $\varphi$-witness in the model $L_\beta^{F_\alpha}(\mathbb{R}, \uB \restrict \alpha)^V$ for some ordinal $\beta < \gamma$, where $\alpha$ is the Wadge rank of $D$. Let $\beta$ be the least ordinal such that there is a $\varphi$-witness in the model $L_\beta^{F_\alpha}(\mathbb{R}, \uB \restrict \alpha)^V$.

Take a set of reals $C \in \uB^V \restrict \alpha$ such that the model $L_\beta^{F_\alpha}(\mathbb{R}, \uB \restrict \alpha)^V$ has a $\varphi$-witness that is ordinal-definable from the parameter $C$ and the predicate $F$, and let $A \in L_\beta^{F_\alpha}(\mathbb{R}, \uB \restrict \alpha)^V$ be the least such $\varphi$-witness in the canonical well-ordering relative to $C$. 

Take a countable hull $\pi_0 \colon P_0 \to V_{\gamma}$ with $(\lambda,C,D) \in \ran(\pi_0)$. For every $\mathbb{R}$-genericity iteration $P_0 \to P_\omega$ we have $\beta \in P_\omega$ by the minimality of $\beta$; otherwise the model $L^{F_\alpha}(\mathbb{R}, \uB \restrict \alpha)^{P_\omega(\mathbb{R}^V)}$ would not be tall enough to reach a $\varphi$-witness, violating the elementarity of $\pi_\omega$.

For every such genericity iteration, the ordinal $\alpha$ is definable from the set $D$ in $P_\omega(\mathbb{R}^V)$ as its Wadge rank, the ordinal $\beta$ is definable from $\alpha$ in $P_\omega(\mathbb{R}^V)$ as the least ordinal such that the model $L_\beta^{F_\alpha}(\mathbb{R}, \uB \restrict \alpha)^{P_\omega(\mathbb{R}^V)}$ contains a $\varphi$-witness, and in turn the least $\varphi$-witness $A$ is definable from $C$ in the model $L_\beta^{F_\alpha}(\mathbb{R}, \uB \restrict \alpha)^{P_\omega(\mathbb{R}^V)}$ by the same definition as in the model $L_\beta^{F_\alpha}(\mathbb{R}, \uB \restrict \alpha)^V$,which is the same model. This implies that $A \in \Hom_{\mathord{<}\lambda}$ by Lemma \ref{lem-absolutely-definable} as desired.

\section{Suslin sets in the derived model}\label{sec-DM-at-limit-of-lt-lambda-strongs}

This section is mostly an exposition of the following theorem regarding the derived model $D(V,\lambda,G)$, which was defined in Remark \ref{oldnewrem}. A proof of the corresponding theorem (also due to Woodin) for the old dervied model $L(\mathbb{R}^*_G, \Hom^*_G)$ can be found in Steel \cite[Section 9]{SteDMT}.




\begin{thm}[Woodin]\label{thm-DM-at-limit-of-lt-lambda-strongs}
 Let $\lambda$ be a limit of Woodin cardinals and of $\mathord{<}\lambda$-strong cardinals. Let $G \subseteq \Col(\omega,\mathord{<}\lambda)$ be a $V$-generic filter. Then the derived model $D(V,\lambda,G)$ satisfies the statement that every set of reals is Suslin.
\end{thm}

\begin{rem}\label{oldnewequivrem}
As shown at the beginning of Section \ref{amodelsec}, $\Hom^{*}_{G}$ is equal to the set of subsets of $\mathbb{R}^{*}_{G}$ in $V(\mathbb{R}^{*}_{G})$ which are Suslin and co-Suslin in $V(\mathbb{R}^{*}_{G})$. It follows from Theorem \ref{thm-DM-at-limit-of-lt-lambda-strongs} then that (assuming its hypotheses)  every set of reals in $D(V, \lambda, G)$ is in $\Hom^{*}_{G}$, so $L(\mathbb{R}^{*}_{G}, \Hom^{*}_{G}) = D(V, \lambda, G)$. 
\end{rem} 

The proof we give for Theorem \ref{thm-DM-at-limit-of-lt-lambda-strongs} is based on one by Steel (unpublished, but see notes by Zeman \cite[pp.~27--28]{ZemNotes}), with a couple of modifications. One modification is that we separate the argument into two parts. The first part is Lemma \ref{lem-suslin-norm-co-suslin} below, which uses only the assumption that $\lambda$ is a limit of $\mathord{<}\lambda$-strong cardinals, and which may be of independent interest. The remainder of the argument will use (in addition to the conclusion of Lemma \ref{lem-suslin-norm-co-suslin}) only the assumption that $\lambda$ is a limit of Woodin cardinals. This large cardinal hypothesis implies that  $\Hom^{*}_{G} \subseteq \Gamma_{+}$ (by the version of the derived model theorem which appears as Theorem 7.1 in \cite{SteDMT}) and that $D(V, \lambda, G)$ satisfies $\AD^{+}$, which is the first part of Theorem 31 of  \cite{WooSEMI}.

The second modification is that we consider Suslin-norms (as defined below) rather than Suslin wellfounded relations as in Steel's proof.  This seems to be necessary in order to split up the proof as we have done.


Given a norm $\varphi$ on a set of reals $A$, the \emph{code} of $\varphi$ is the pair of binary relations $(R_\le, R_<)$ defined by
\begin{align*}
 R_\le &= \{(x,y) \in \omega^\omega \times \omega^\omega : x \in A \wedge (y \in A \rightarrow \varphi(x) \le \varphi(y))\}\\
 R_<   &= \{(x,y) \in \omega^\omega \times \omega^\omega : x \in A \wedge (y \in A \rightarrow \varphi(x) < \varphi(y))\}.
\end{align*}
Note that the property of being a code for some norm on $A$ can be expressed as the conjunction of the following conditions:
 \begin{itemize}
  \item $R_\le \cap (A \times A)$ is a prewellordering on $A$,
  \item $R_< \cap (A \times A)$ is the strict part of this prewellordering,
  \item $(A \times \neg A) \subseteq R_\le \subseteq (A \times \omega^\omega)$ and
  \item $(A \times \neg A) \subseteq R_< \subseteq (A \times \omega^\omega)$.
 \end{itemize}
Two norms on a set $A$ are said to be \emph{equivalent} if their codes are the same.  Every norm is equivalent to a unique \emph{regular norm}, a norm whose range is an ordinal.
An \emph{initial segment} of a set $A$ under a norm $\varphi$ on $A$ is a set of the form $\{x \in A : \varphi(x) < \xi\}$ for some ordinal $\xi$.

 We will need the following fact about extending norms.  The proof is standard and is left to the reader.
 
 \begin{lem}\label{lem-extending-norms}
  Let $\varphi$ be a regular norm on a set of reals $A$ and let $\varphi'$ be a regular norm on a set of reals $A'$.
  Let  $(R_\le, R_<)$ and $(R'_\le, R'_<)$ denote the pairs of relations coding $\varphi$ and $\varphi'$ respectively.
  If $A \subseteq A'$, $R_\le \subseteq R_\le'$, and $R_< \subseteq R_<'$,
  then the set $A$ is an initial segment of the set $A'$ under its norm $\varphi'$, and we have $\varphi = \varphi' \restrict A$.
 \end{lem}
 
%

If $\Gamma$ is a pointclass, a $\Gamma$-\emph{norm} on a set of reals $A \in \Gamma$ is a norm $\varphi$ on $A$ with the property that its code $(R_\le,R_<)$ is in $\Gamma$, meaning that both relations $R_\le$ and $R_<$ are continuous preimages of sets in $\Gamma$. This is equivalent to the usual definition \cite[p.~153]{MosDST2009} if $\Gamma$ is an adequate pointclass.  All of the pointclasses that we will consider are adequate.

\begin{rem}\label{initsegrem}
For pointclasses $\Gamma$ closed under unions, intersections and continuous preimages (as all the pointclasses we consider are) if $\varphi'$ is a $\Gamma$-norm on a set $A' \in \Gamma$, then every proper initial segment $A$ of $A'$ under $\varphi'$ is in $\Delta = \Gamma \cap \check{\Gamma}$, and the corresponding restricted norm $\varphi' \restrict A$ is a $\Delta$-norm (see Remark 4.13 of \cite{Larson:Extensions}, for instance).
\end{rem}


By a \emph{Suslin-norm} we mean a $\Gamma$-norm where $\Gamma$ is the pointclass of Suslin sets. Given a cardinal $\lambda$, a $\uB_{\lambda}$-\emph{norm} is a $\Gamma$-norm where $\Gamma$ is the pointclass of $\lambda$-universally Baire sets.


For a pointset $R$, a $\posSigmaoneone(R)$-\emph{statement} 
is one of the form 
\[\exists y\,\big(Q(y) \And \forall n < \omega\, (y)_n \in R)\]
where $Q$ is a $\Sigma^1_1$ formula and we think of a real $y$ as coding a sequence of reals $(y)_0,(y)_1,(y)_2,\ldots$ in some recursive way (see, for instance, Moschovakis \cite[Lemma 7D.7]{MosDST2009} or page 21 of \cite{Larson:Extensions}).
For pointsets $R_1,\ldots,R_n$, we define $\posSigmaoneone(R_1,\ldots,R_n)$ as $\posSigmaoneone(R)$ where $R$ is the disjoint union of the sets $R_1,\ldots,R_n$.

If $R$ is a binary relation then the statement that $R$ is illfounded can be expressed by a $\posSigmaoneone(R)$ statement, and the statement that $R$ is not a prewellordering can be expressed by a  $\posSigmaoneone(R,\neg R)$ statement.
In addition, one can say (quite trivially) that $A \cap B \ne \emptyset$ with a $\posSigmaoneone(A,B)$ statement.

For any trees $T_1,\ldots,T_n$,
by $\posSigmaoneone(\text{p}[T_1],\ldots,\text{p}[T_n])$ \emph{generic absoluteness}
we mean the fact that, for any generic extension $V[g]$, any given $\posSigmaoneone$ statement about the sets $\p[T_1]^{V[g]},\ldots,\p[T_n]^{V[g]}$ holds in $V[g]$ if and only if the corresponding $\posSigmaoneone$ statement about the sets $\p[T_1],\ldots,\p[T_n]$ holds in $V$.  
This fact is easily proved using the absoluteness of wellfoundedness for a certain tree built from $T_1,\ldots,T_n$ and a tree for a $\posSigmaoneone$ set.  (There is nothing special about the case of generic extensions versus arbitrary outer models here, but it is the only case we will need.)
This generic absoluteness property encompasses many of the common absoluteness arguments involving trees.
It can be used to show, for example, that if $R$ is a universally Baire wellfounded relation (or prewellordering) then the canonical expansion $R^{V[g]}$ is a universally Baire wellfounded relation (or prewellordering).
It can also be used to show that the canonical expansions of universally Baire sets are unique (which we did in Lemma \ref{lem-unique-expansion}). 

The following lemma is the first part of the proof of Theorem \ref{thm-DM-at-limit-of-lt-lambda-strongs}, as outlined above.

\begin{lem}\label{lem-suslin-norm-co-suslin}
 Let $\lambda$ be a limit of $\mathord{<}\lambda$-strong cardinals and let $G \subseteq \Col(\omega,\mathord{<}\lambda)$ be a $V$-generic filter.  Then in the symmetric extension $V(\mathbb{R}^*_G)$, every Suslin set that admits a Suslin-norm is co-Suslin.
 \end{lem}
 
\begin{proof}
 Let $A$ be a Suslin set in $V(\mathbb{R}^*_G)$ that admits a regular Suslin-norm $\varphi$, and let $R_\le$ and $R_<$ be the Suslin relations coding $\varphi$. Let $T$, $T_\le$, and $T_<$ be trees in $V(\bbR^{*}_{G})$ projecting to $A$, $R_\le$, and $R_<$ respectively.  By the usual homogeneity argument $T$, $T_\le$, and $T_<$ are in $V[G \restrict \xi]$ for some $\xi < \lambda$. Let $\kappa < \lambda$ be a $\less\lambda$-strong cardinal in $V[G \restrict \xi]$. 
 
 In any generic extension of $V[G \restrict \xi]$ by a poset of size less than $\lambda$, the pair of relations $(\p[T_\le], \p[T_<])$ codes a norm on the set $\p[T]$, since (for instance) the properties of $(A, R_\le, R_<)$ in $V(\mathbb{R}^*_G)$ that ensure coding a norm are preserved under restriction to the reals of a smaller universe.
 
  Let $\alpha = 2^{2^\kappa}$. Define the set $A_0 = \p[T]^{V[G \restrict \alpha]}$ and the pair of relations $(R^{0}_{\le}, R^{0}_{<}) = (\p[T_\le], \p[T_<])^{V[G \restrict \alpha]}$. Then the pair $(R^{0}_{\le}, R^{0}_{<})$ codes a norm $\varphi_0$ on the set $A_0$.
 
 
 \begin{claim*}
  In $V[G \restrict \alpha]$,
  the set $A_0$ is $\lambda$-universally Baire and the norm $\varphi_0$ on $A_0$ is a $\uB_{\lambda}$-norm.
 \end{claim*}
 
 \begin{proof}
  Fix $\eta < \lambda$. We will show that the set $A_0$ is $\eta$-universally Baire and the norm $\varphi_0$ on $A_0$ is a $\uB_{\eta}$-norm.
  By increasing $\eta$ we may assume that $\alpha < \eta$ and $\left|V_\eta\right| = \eta$.
  Take an elementary embedding $j \colon V[G \restrict \xi] \to M$ with critical point $\kappa$, where $M$ is transitive, $M^\omega \subseteq M$, $V_\eta[G \restrict \xi] \subseteq M$, and $\eta < j(\kappa)$. 
   
  
  Let $A_\eta = \p[j(T)]^{V[G \restrict \alpha]}$, $R_{\le}^{\eta} = \p[j(T_\le)]^{V[G \restrict \alpha]}$ 
  and $R_{<}^{\eta} = \p[j(T_<)]^{V[G \restrict \alpha]}$.
  Then $A_\eta$, $R_{\le}^{\eta}$ and $R_{<}^{\eta}$ are $\eta$-universally Baire in $V[G \restrict \alpha]$ by a theorem of Woodin (see Steel \cite[Theorem 4.5]{SteDMT}.)
  
  In any generic extension of $M$ by a poset of size less than $j(\lambda)$, the pair of relations $(\p[j(T_\le)], \p[j(T_<)])$ codes a norm on the set $\p[j(T)]$, by the elementarity of $j$.
  Since $\alpha < j(\lambda)$ and the models $M[G \restrict \alpha]$ and $V[G \restrict \alpha]$ have the same reals, it follows that the pair $(R_{\le}^{\eta}, R_{<}^{\eta})$ codes a regular norm $\varphi_\eta$ on the set $A_\eta$.
  We have then that $\varphi_\eta$ is a $\uB_{\eta}$-norm on $A_\eta$ in $V[G \restrict \alpha]$.
 
  Using $j$ to map branches of trees pointwise, we get $A_0 \subseteq A_\eta$, $R_{\le}^{0} \subseteq R_{\le}^{\eta}$, and $R_{<}^{0} \subseteq R_{<}^{\eta}$.
  By Lemma \ref{lem-extending-norms},
  these inclusions imply that the set $A_0$ is an initial segment of the set $A_\eta$ under its norm $\varphi_\eta$, and that the norm $\varphi_0$ on $A_0$ is the restriction of $\varphi_\eta$ to this initial segment.
  By Remark \ref{initsegrem}, in $V[G \restrict \alpha]$, the set $A_0$ is $\eta$-universally Baire and the norm $\varphi_0$ on $A_0$ is a $\uB_{\eta}$-norm. \end{proof}
  

 Now we complete the proof of Lemma \ref{lem-suslin-norm-co-suslin}. By the claim, we may fix $\lambda$-absolutely complementing pairs of trees $(T^*,\tilde{T})$, $(T_\le^*, \tilde{T}_\le)$, and $(T_<^*,\tilde{T}_<)$ in $V[G \restrict \alpha]$ for $A_0$, $R_{\le}^{0}$, and $R_{<}^{0}$ respectively, and
 define the canonical expansions $A^* = \p[T^*]^{V(\mathbb{R}^*_G)}$, $R_\le^* = \p[T_\le^*]^{V(\mathbb{R}^*_G)}$, and $R_<^* =  \p[T_<^*]^{V(\mathbb{R}^*_G)}$.
 
 Because the pair of relations $(R_{\le}^{0}, R_{<}^{0})$ codes a norm on the set $A_0$,
 it follows that the pair of relations $(R_\le^*, R_<^*)$ codes a regular norm $\varphi^*$ on the set $A^*$
 by generic absoluteness for 
 $\Sigma^1_1(\p[T^*], \p[\tilde{T}], \p[T_\le^*], \p[\tilde{T}_\le], \p[T_<^*], \p[\tilde{T}_<])$.
 
 
 
 We have $A \subseteq A^*$ by Lemma \ref{lem-projection-inclusion}, which can be seen as an instance of $\Sigma^1_1(\p[T], \p[\tilde{T}])$ generic absoluteness.
 Similarly $R_\le \subseteq R_\le^*$ and $R_< \subseteq R_<^*$.
 By Lemma \ref{lem-extending-norms} again, these inclusions imply that  the set $A$ is an initial segment of the set $A^*$ under the norm $\varphi^*$.  In $V(\mathbb{R}^*_G)$ the norm $\varphi^*$ is an $\ScS$-norm where $\ScS$ denotes the pointclass of Suslin co-Suslin sets, so its inital segment $A$ is a Suslin co-Suslin set as in Remark \ref{initsegrem}, as desired.
\end{proof}

\begin{rem}
	In \cite{SchindlerYasuda}, a pointclass $\Gamma$
	consisting of universally Baire sets of reals is defined to be \emph{productive} if it is closed under
	complements, projections, and satisfies that statement that, for all $A \subseteq \omega^{\omega} \times \omega^{\omega}$ in (or continuously reducible to a member of) $\Gamma$, in all set-generic forcing extensions $V[H]$ the sets $\{ y : \exists x \in \omega^{\omega} (x,y) \in A^{V[H]} \}$ and 
	$(\{y \in \omega^{\omega} : \exists x \in \omega^{\omega} (x,y) \in A\})^{V[H]}$ are equivalent. An argument very similar to the proof of Lemma \ref{lem-suslin-norm-co-suslin} shows that the collection of universally Baire sets of reals is productive in $V(\bbR^{*}_{G})$, assuming the hypotheses of the lemma. We conjecture that this fact (and in fact a stronger fact) already follows from just the theory $\AD^{+}$ +``All sets of reals are universally Baire". 

    More specifically, we conjecture that if $V$ is a model of $\AD^++``$All sets of reals are universally Baire" then for every set of reals $A$, in all set generic extensions $V[G*H]$, there is an elementary embedding 
    \begin{center} $j: L(A^{V[G]}, \mathbb{R}^{V[G]})\rightarrow L(A^{V[G*H]}, \mathbb{R}^{V[G*H]})$ \end{center}
    with $j(A^{V[G]})=A^{V[G*H]}$.
\end{rem} 

Before proceeding with the proof of Theorem \ref{thm-DM-at-limit-of-lt-lambda-strongs},
we make a remark about the necessity of assuming in Lemma \ref{lem-suslin-norm-co-suslin} that the Suslin set of reals $A$ admits a Suslin-norm.

\begin{rem}\label{rem-woodin-counterexample}
 If $\lambda$ is a Woodin cardinal then Lemma \ref{lem-suslin-norm-co-suslin} holds even without the hypothesis that the set has a Suslin-norm: every Suslin set in $V(\mathbb{R}^*_G)$ is co-Suslin by a theorem of Woodin (see Larson \cite[Theorem 1.5.12]{LarStationaryTower} or Steel \cite[Theorem 4.4]{SteDMT}). However, in our application we do not want to make the assumption that $\lambda$ is a Woodin cardinal; together with the assumption that it is a limit of Woodin cardinals, this would be significantly stronger than the hypothesis of the main theorem.
 
 On the other hand, if $\lambda$ is singular then the hypothesis of the existence of a Suslin-norm in Lemma \ref{lem-suslin-norm-co-suslin} is necessary: in $V(\mathbb{R}^*_G)$, there is a Suslin set of reals that is not co-Suslin. An example was pointed out to us by Woodin \cite{WooSuslinNotCoSuslinEmail}.
 We generalize this example to show that if $\mathsf{ZF}$ holds and $\omega_1$ is singular (as in the symmetric model obtained by the Levy collapse at a singular cardinal) then there is a Suslin set of reals that is not co-Suslin.
 
 Take a sequence of countable ordinals $\langle \alpha_n : n < \omega\rangle$ that is cofinal in $\omega_1$ and let $A$ be the set of all reals of the form $n^\frown x$, where $n < \omega$ and $x$ is a real coding a wellordering of $\omega$ of order type at most $\alpha_n$.
 This set $A$ is Suslin, as witnessed by the tree of attempts to build a real $n^\frown x$ and a function $f : \omega \to \alpha_n$ such that $x$ codes a linear ordering on $\omega$ and $f$ is strictly increasing with respect to this linear ordering.
 
 Assume toward a contradiction that the set $A$ is co-Suslin.
 The set of all reals $n^\frown x$ where $n < \omega$ and $x$ is a real coding \emph{any} wellordering of $\omega$
 is also Suslin, so the intersection of this set with the complement of $A$ is Suslin.
 
 In other words, the set of all reals $n^\frown x$ where $n < \omega$ and $x$ is a real coding a wellordering of $\omega$ of order type greater than $\alpha_n$ is Suslin.
 Then by considering leftmost branches we can choose, for each $n < \omega$, a real $x_n$ coding a wellordering of $\omega$ of order type greater than $\alpha_n$.  This is a contradiction because it allows us to define a wellordering of $\omega$ of order type $\omega_1$.

\end{rem}


\begin{proof}[Proof of Theorem \ref{thm-DM-at-limit-of-lt-lambda-strongs}]
 Working in the derived model $D(V,\lambda,G)$, suppose toward a contradiction that not every set of reals is Suslin. Since $\mathsf{AD}^+$ holds, the set of Suslin cardinals is closed in $\Theta$ (see Ketchersid \cite{KetMoreAD} or Larson \cite{Larson:Extensions} for a proof) so there is a largest Suslin cardinal $\kappa$ (by Corollary 6.19 of \cite{Larson:Extensions}). The pointclass of $\kappa$-Suslin sets (equivalently, of all Suslin sets) is non-selfdual, and it has the prewellordering property, meaning that every Suslin set has a Suslin-norm (see Jackson \cite[Lemma 3.6]{JacHandbook} and Remark 6.6 of \cite{Larson:Extensions}). 
 So we can fix a complete Suslin set $A$ of $D(V,\lambda,G)$ and note that a Suslin-norm on $A$ exists.

By Lemma \ref{lem-suslin-norm-co-suslin},
the set $A$ is Suslin and co-Suslin in $V(\mathbb{R}^*_G)$, which means that $A$ is in $\Hom^*_G$. We can fix then a $\xi < \lambda$ and $\lambda$-absolutely complementing trees $T_{A}$ and $\tilde{T}_{A}$ on $\omega \times \Ord$ in $V[G \restrict \xi]$ such that $A = \p[T_{A}]^{V(\bbR^{*}_{G})}$. Letting $A_{0} = \p[T_{A}]^{V[G \restrict \xi]}$, we have then that $A = A_0^{V(\mathbb{R}^*_G)}$. 
This means that 
$V[G\restrict \xi] \models \omega^\omega \setminus A_0 \in \uB_{\lambda}$ as well, so the set $\omega^\omega \setminus A_0$ has a $\uB_{\lambda}$-semiscale $\vec{\varphi}$ in $V[G \restrict \xi]$ by Steel \cite[Theorem 5.3]{SteDMT}. (This result uses the fact that $\uB_{\lambda} = \Hom_{\mathord{<}\lambda}$ because $\lambda$ is a limit of Woodin cardinals.  Also, the semiscale obtained by Steel is a scale, but we do not need the additional lower semicontinuity property of scales here.)
 
 As in Section \ref{semiscalesec}, to say that $\vec{\varphi}$ is a $\uB_{\lambda}$-semiscale means that
 its component norms $\varphi_n$ are $\uB_{\lambda}$-norms, uniformly in $n$.  (In fact the uniformity is automatic here because the pointclass $\uB_{\lambda}$ is closed under countable unions.)
 More precisely, the relation $R_0$ coding the semiscale $\vec{\varphi}$, defined by
 \[R_0 = \big\{ (\bar{n}, x, y) \in \omega^{\omega} \times (\omega^\omega \setminus A_0) \times (\omega^\omega \setminus A_0) : \varphi_n(x) \le \varphi_n(y) \big\}, \]
 is $\lambda$-universally Baire, where $\bar{n}$ indicates the constant function from $\omega$ to $n$.
Fix trees $T_{R}$ and $\tilde{T}_{R}$ in $V[G \restrict \xi]$ witnessing this, and let $R$ denote the expanded relation $R_{0}^{V(\bbR^{*}_{G})}$. 
 
 
The expanded relation $R$ codes a semiscale on the set $\omega^\omega \setminus A$, which was our universal co-Suslin set of the derived model.
One way to see this is to use the fact that $(\HC^{V[G \restrict \xi]}; \mathord{\in}, A_0,R_0) \prec (\HC^*_G; \mathord{\in}, A, R)$ because $\lambda$ is a limit of Woodin cardinals
 (see Steel \cite[Lemma 7.3]{SteDMT}.)
 Another way 
 is to apply $\posSigmaoneone(\p[T_A],\p[\tilde{T}_A],\p[T_R],\p[\tilde{T}_R])$ generic absoluteness.  
  
 Because the relation $R$ is in $\Hom^*_G$, it is in the derived model $D(V,\lambda,G)$.
 In $D(V,\lambda,G)$, the tree of the semiscale coded by $R$ witnesses that the set $\omega^\omega \setminus A$ is Suslin, i.e., that the set $A$ is co-Suslin, giving a contradiction.
\end{proof}


Finally, we establish the last remaining component required for the proof of the main theorem:

\begin{lem}\label{lem-absorption-of-uB}
Let $\lambda$ be a limit of Woodin cardinals and a limit of strong cardinals, and  let $G \subseteq \Col(\omega,\mathord{<}\lambda)$ be a $V$-generic filter.
Then every set of reals in $\Hom^*_G$ is universally Baire in $L^F(\mathbb{R}^*_G, \Hom^*_G)^{V(\mathbb{R}^*_G)}$.
\end{lem}

\begin{proof}
Let $A\in \Hom^*_G$, which means that $\omega^\omega \setminus A \in \Hom^*_G$ as well.
As in the proof of Theorem \ref{thm-DM-at-limit-of-lt-lambda-strongs},
it follows from Steel \cite[Theorem~5.3]{SteDMT} (using the fact that $\lambda$ is a limit of Woodin cardinals) that $A$ and its complement admit $\Hom^*_G$-semiscales.
By Lemma \ref{lem-ordinal-ac-equivalent} and the discussion preceding it, the fact that $\lambda$ is a limit of strong cardinals implies that every set in $\Hom^*_G$ is universally Baire in $V(\mathbb{R}^*_G)$.
So in $V(\mathbb{R}^*_G)$ the set $A$ is universally Baire and there are $\uB$-semiscales on $A$ and $\omega^\omega \setminus A$. By Theorem \ref{FreflectsuBthrm}, then, with $\Gamma = \Hom^{*}_{G}$, $A$ is universally Baire in $L^F(\mathbb{R}^*_G, \Hom^*_G)^{V(\mathbb{R}^*_G)}$. 
\end{proof}

\section*{Acknowledgements}
The authors thank John Steel for permission to include
an adaptation of his unpublished argument for Theorem \ref{thm-DM-at-limit-of-lt-lambda-strongs}, and Hugh Woodin for pointing out the result mentioned in Remark \ref{rem-woodin-remark}.

\bibliography{math}{}

\begin{thebibliography}{10}

\bibitem{CLSSSZSquares}
Andres Caicedo, Paul Larson, Grigor Sargsyan, Ralf Schindler, John Steel, and
  Martin Zeman.
\newblock Square principles in $\mathbb{P}_\text{max}$ extensions.
\newblock {\em arXiv preprint arXiv:1205.4275}, 2012.

\bibitem{FenMagWoo}
Qi~Feng, Menachem Magidor, and W.~Hugh Woodin.
\newblock Universally {B}aire sets of reals.
\newblock In {\em Set theory of the continuum}, pages 203--242. Springer, 1992.

\bibitem{JacHandbook}
Stephen~C. Jackson.
\newblock Structural consequences of {AD}.
\newblock In Matthew Foreman and Akihiro Kanamori, editors, {\em Handbook of
  Set Theory}, volume~3. Springer, 2010.

\bibitem{Jec2002}
Thomas Jech.
\newblock {\em Set Theory}.
\newblock Springer Monographs in Mathematics. Springer--Verlag, third
  millennium edition, 2002.

\bibitem{KanHigherInfinite}
Akihiro Kanamori.
\newblock {\em The Higher Infinite: Large Cardinals in Set Theory from Their
  Beginnings}.
\newblock Springer Monographs in Mathematics. Springer, 2003.

\bibitem{KKMW}
Alexander~S. Kechris, Eugene~M. Kleinberg, Yiannis~N. Moschovakis, and W.~Hugh
  Woodin.
\newblock The axiom of determinacy, strong partition properties and nonsingular
  measures.
\newblock In {\em Games, scales, and Suslin cardinals. The Cabal Seminar, Vol.
  I. Reprints of papers and new material based on the Los Angeles Caltech-UCLA
  Logic Cabal Seminar 1976--1985}, pages 333--354. Cambridge: Cambridge
  University Press; Urbana, IL: Association for Symbolic Logic (ASL), 2008.

\bibitem{KecMosScales}
Alexander~S. Kechris and Yiannis~N. Moschovakis.
\newblock Notes on the theory of scales.
\newblock In Alexander~S. Kechris, Benedikt L\"owe, and John~R. Steel, editors,
  {\em Games, Scales and {S}uslin Cardinals: The Cabal Seminar, Volume {I}},
  volume~31 of {\em Lecture Notes in Logic}, pages 28--74. Cambridge University
  Press, Cambridge, 2008.

\bibitem{KetMoreAD}
Richard Ketchersid.
\newblock More structural consequences of $\mathrm{AD}$.
\newblock In L.~Babinkostova, A.E. Caicedo, S.~Geschke, and M.~Scheepers,
  editors, {\em Set Theory and its Applications: Annual Boise Extravaganza in
  Set Theory}, volume 533 of {\em Contemporary Mathematics}, pages 71--105.
  American Mathematical Society, 2011.

\bibitem{LarStationaryTower}
Paul~B. Larson.
\newblock {\em The stationary tower: Notes on a course by {W}.~{H}ugh
  {W}oodin}, volume~32 of {\em University Lecture Series}.
\newblock American Mathematical Society, 2004.

\bibitem{Larson:Extensions}
Paul~B. Larson.
\newblock {\em Extensions of the Axiom of Determinacy}, volume~78 of {\em
  University Lecture Series}.
\newblock American Mathematical Society, 2023.

\bibitem{MarSteProjDet}
Donald~A Martin and John~R Steel.
\newblock A proof of projective determinacy.
\newblock {\em Journal of the American Mathematical Society}, 2(1):71--125,
  1989.

\bibitem{MarSteIterationTrees}
Donald~A. Martin and John~R. Steel.
\newblock Iteration trees.
\newblock {\em Journal of the American Mathematical Society}, pages 1--73,
  1994.

\bibitem{MarWooWeaklyHom}
Donald~A. Martin and W.~Hugh Woodin.
\newblock Weakly homogeneous trees.
\newblock In Alexander~S. Kechris, Benedikt L\"owe, and John~R. Steel, editors,
  {\em Games, Scales and {S}uslin Cardinals: The Cabal Seminar, Volume {I}},
  volume~31 of {\em Lecture Notes in Logic}, pages 421--438. Cambridge
  University Press, Cambridge, 2008.

\bibitem{MosDST2009}
Yiannis~N. Moschovakis.
\newblock {\em Descriptive Set Theory}, volume 155 of {\em Mathematical Surveys
  and Monographs}.
\newblock American Mathematical Society, Providence, Rhode Island, second
  edition, 2009.

\bibitem{Muller:consistency}
Sandra M\"{u}ller.
\newblock The consistency strength of determinacy when all sets are universally
  {B}aire.
\newblock 2024.
\newblock preprint.

\bibitem{NeeOptimalII}
Itay Neeman.
\newblock Optimal proofs of determinacy \textsc{II}.
\newblock {\em The Journal of Mathematical Logic}, 2(2):pp. 227--258, 2002.

\bibitem{NeeHandbook}
Itay Neeman.
\newblock Determinacy in $\textsc{L}(\mathbb{R})$.
\newblock In Matthew Foreman and Akihiro Kanamori, editors, {\em Handbook of
  Set Theory}, pages 1887--1950. Springer Netherlands, 2010.

\bibitem{LSA}
Grigor Sargsyan and Nam Trang.
\newblock {\em The {Largest} {Suslin} {Axiom}}, volume~56 of {\em Lect. Notes
  Log.}
\newblock Cambridge: Cambridge University Press, 2024.

\bibitem{Schindler:SetTheory}
Ralf Schindler.
\newblock {\em Set theory}.
\newblock Universitext. Springer, Cham, 2014.
\newblock Exploring independence and truth.

\bibitem{SchindlerYasuda}
Ralf Schindler and Taichi Yasuda.
\newblock Martin's maximum${}^{\ast, ++}_{\mathfrak{c}}$ in $\mathbb{P}_{\max}$
  extensions of strong models of determinacy, 2024.

\bibitem{SteSTFree}
John~R. Steel.
\newblock A stationary-tower-free proof of the derived model theorem.
\newblock In Su~Gao, Steve Jackson, and Yi~Zhang, editors, {\em Advances in
  Logic}, volume 425 of {\em Contemporary Mathematics}, pages 1--8. American
  Mathematical Society, 2007.

\bibitem{SteDMT}
John~R. Steel.
\newblock The derived model theorem.
\newblock In {\em Logic Colloquium 2006}, volume~32 of {\em Lecture Notes in
  Logic}, pages 280--327. Cambridge University Press, 2009.

\bibitem{WooSEMI}
W.~Hugh Woodin.
\newblock Suitable extender models {I}.
\newblock {\em Journal of Mathematical Logic}, 10(1,2):101--339, 2010.

\bibitem{WooSuslinNotCoSuslinEmail}
W.~Hugh Woodin.
\newblock Personal communication, November 26, 2014.

\bibitem{ZemNotes}
Martin Zeman.
\newblock Untitled.
\newblock
  \url{http://wwwmath.uni-muenster.de/logik/Personen/rds/notes_by_zeman_1.pdf},
  2010.
\newblock Notes from John Steel's lectures at the 2010 Conference on the Core
  Model Induction and Hod Mice at WWU M\"unster.

\end{thebibliography}
\bibliographystyle{plain}

\end{document}